\documentclass[a4paper, twoside, 11pt, english]{article}

\usepackage[T1]{fontenc}
\usepackage[utf8]{inputenc}

\usepackage{etoolbox}

\usepackage{amsmath,amsthm,amssymb,stmaryrd}%\usepackage{upgreek}
\usepackage{mathtools}
\usepackage[ttscale=.875]{libertine}
\usepackage[libertine]{newtxmath}

\usepackage[numbers,sort&compress]{natbib}
\usepackage{comment}
\usepackage{ifthen}
\usepackage{algorithm2e}
\usepackage{multicol}

\usepackage{fancyhdr, titlesec, url, enumerate, microtype,setspace}

\usepackage[all,knot,poly]{xy}
\usepackage{tikz}
\usetikzlibrary{decorations.pathreplacing}

\usepackage{tikz-qtree,varwidth}
\usepackage{youngtab}
\usepackage{ytableau}

\usepackage[english]{babel}

\usepackage{hyperref}

\CompileMatrices

\usetikzlibrary{calc}

\newcount\hh
\newcount\mm
\mm=\time
\hh=\time
\divide\hh by 60
\divide\mm by 60
\multiply\mm by 60
\mm=-\mm
\advance\mm by \time
\def\hhmm{\number\hh:\ifnum\mm<10{}0\fi\number\mm}

\newcommand{\periodafter}[1]{\ifstrempty{#1}{}{#1.}}
\titleformat{\section}[block]{\scshape\filcenter\LARGE\boldmath}{\thesection.}{.5em}{}
\titleformat{\subsection}[block]{\bfseries\filcenter\large\boldmath}{\thesubsection.}{.5em}{\medskip}
\titleformat{\subsubsection}[runin]{\bfseries\boldmath}{\thesubsubsection.}{.5em}{\periodafter}%{}[.]
\titlespacing{\subsubsection}{0pt}{\topsep}{.5em}

\newtheoremstyle{ntheorem}%
	{\topsep}{\topsep}{\itshape}{0pt}{\bfseries}{.}{.5em}%
	{\thmnumber{#2.\hspace{.5em}}\thmname{#1}\thmnote{ (#3)}}
	
\newtheoremstyle{ndefinition}%
	{\topsep}{\topsep}{\normalfont}{0pt}{\bfseries}{.}{.5em}%
	{\thmnumber{#2.\hspace{.5em}}\thmname{#1}\thmnote{ (#3)}}
	
\newtheoremstyle{nremark}%
	{\topsep}{\topsep}{\normalfont}{0pt}{\itshape}{.}{.5em}%
	{\thmnumber{}\thmname{#1}\thmnote{ (#3)}}

\theoremstyle{ntheorem}
  	\newtheorem{theorem}[subsubsection]{Theorem}
  	
	\newtheorem{lemma}[subsubsection]{Lemma}
  	\newtheorem{corollary}[subsubsection]{Corollary}

\theoremstyle{ndefinition}

\makeatletter
\def\@equationname{equation}
\newenvironment{eqn}[1]{%
    \def\mymathenvironmenttouse{#1}%
    \ifx\mymathenvironmenttouse\@equationname%
        \refstepcounter{subsubsection}%
    \else
        \patchcmd{\@arrayparboxrestore}{equation}{subsubsection}{}{}%          doesn't change output?
        \patchcmd{\print@eqnum}{equation}{subsubsection}{}{}%
        \patchcmd{\incr@eqnum}{equation}{subsubsection}{}{}%
    \fi
    \csname\mymathenvironmenttouse\endcsname%
}{%
    \ifx\mymathenvironmenttouse\@equationname%
        \tag{\thesubsubsection}%
    \fi
    \csname end\mymathenvironmenttouse\endcsname%
}
\makeatother

\fancypagestyle{plain}{
\fancyhf{}\fancyfoot[LC,RC]{}
\fancyfoot[LE,RO]{$\thepage$}

}

\UseTips
\SelectTips{eu}{11}

\newdir{ >}{{}*!/-10pt/@{>}}
\newdir{ -}{{}*!/-10pt/@{}}
\newdir{> }{{}*!/+10pt/@{>}}

\makeatletter

\xyletcsnamecsname@{dir4{}}{dir{}}
\xydefcsname@{dir4{-}}{\line@ \quadruple@\xydashh@}
\xydefcsname@{dir4{.}}{\point@ \quadruple@\xydashh@}
\xydefcsname@{dir4{~}}{\squiggle@ \quadruple@\xybsqlh@}
\xydefcsname@{dir4{>}}{\Tttip@}
\xydefcsname@{dir4{<}}{\reverseDirection@\Tttip@}

\xydef@\quadruple@#1{%
	\edef\Drop@@{%
		\dimen@=#1\relax
		\dimen@=.5\dimen@
		\A@=-\sinDirection\dimen@
		\B@=\cosDirection\dimen@
		\setboxz@h{%
			\setbox2=\hbox{\kern3\A@\raise3\B@\copy\z@}%
			\dp2=\z@ \ht2=\z@ \wd2=\z@ \box2
			\setbox2=\hbox{\kern\A@\raise\B@\copy\z@}%
			\dp2=\z@ \ht2=\z@ \wd2=\z@ \box2
			\setbox2=\hbox{\kern-\A@\raise-\B@\copy\z@}%
			\dp2=\z@ \ht2=\z@ \wd2=\z@ \box2
			\setbox2=\hbox{\kern-3\A@\raise-3\B@ \noexpand\boxz@}%
			\dp2=\z@ \ht2=\z@ \wd2=\z@ \box2
		}%
		\ht\z@=\z@ \dp\z@=\z@ \wd\z@=\z@ \noexpand\styledboxz@
	}%
}

\xydef@\Tttip@{\kern2pt \vrule height2pt depth2pt width\z@
	\Tttip@@ \kern2pt \egroup
	\U@c=0pt \D@c=0pt \L@c=0pt \R@c=0pt \Edge@c={\circleEdge}%
	\def\Leftness@{.5}\def\Upness@{.5}%
	\def\Drop@@{\styledboxz@}\def\Connect@@{\straight@{\dottedSpread@\jot}}}
	
\xydef@\Tttip@@{%
	\dimen@=.25\dimen@
 	\B@=\cosDirection\dimen@
	\setboxz@h\bgroup\reverseDirection@\line@ \wdz@=\z@ \ht\z@=\z@ \dp\z@=\z@
	{\vDirection@(1,-1)\xydashl@ \xyatipfont\char\DirectionChar}%
	{\vDirection@(1,+1)\xydashl@ \xybtipfont\char\DirectionChar}%
}

\xydef@\ar@form{
	\ifx \space@\next \expandafter\DN@\space{\xyFN@\ar@form}%
	\else\ifx ^\next \DN@ ^{\xyFN@\ar@style}\edef\arvariant@@{\string^}%
	\else\ifx _\next \DN@ _{\xyFN@\ar@style}\edef\arvariant@@{\string_}%
	\else\ifx 0\next \DN@ 0{\xyFN@\ar@style}\def\arvariant@@{0}%
	\else\ifx 1\next \DN@ 1{\xyFN@\ar@style}\def\arvariant@@{1}%
	\else\ifx 2\next \DN@ 2{\xyFN@\ar@style}\def\arvariant@@{2}%
	\else\ifx 3\next \DN@ 3{\xyFN@\ar@style}\def\arvariant@@{3}%
	\else\ifx 4\next \DN@ 4{\xyFN@\ar@style}\def\arvariant@@{4}%
	\else\ifx \bgroup\next \let\next@=\ar@style
	\else\ifx [\next \DN@[##1]{\ar@modifiers{[##1]}}%]
	\else\ifx *\next \DN@ *{\ar@modifiers}%
	\else\addLT@\ifx\next \let\next@=\ar@slide
	\else\ifx /\next \let\next@=\ar@curveslash
	\else\ifx (\next \let\next@=\ar@curveinout %)
	\else\addRQ@\ifx\next \addRQ@\DN@{\ar@curve@}%
	\else\addLQ@\ifx\next \addLQ@\DN@{\xyFN@\ar@curve}%
	\else\addDASH@\ifx\next \addDASH@\DN@{\defarstem@-\xyFN@\ar@}%
	\else\addEQ@\ifx\next \addEQ@\DN@{\def\arvariant@@{2}\defarstem@-\xyFN@\ar@}%
	\else\addDOT@\ifx\next \addDOT@\DN@{\defarstem@.\xyFN@\ar@}%
	\else\ifx :\next \DN@:{\def\arvariant@@{2}\defarstem@.\xyFN@\ar@}%
	\else\ifx ~\next \DN@~{\defarstem@~\xyFN@\ar@}%
	\else\ifx !\next \DN@!{\dasharstem@\xyFN@\ar@}%
	\else\ifx ?\next \DN@?{\ar@upsidedown\xyFN@\ar@}%
	\else \let\next@=\ar@error
	\fi\fi\fi\fi\fi\fi\fi\fi\fi\fi\fi\fi\fi\fi\fi\fi\fi\fi\fi\fi\fi\fi\fi \next@}

\makeatother

\newcommand{\fl}{\rightarrow}

\newcommand{\dfl}{\Rightarrow}
\newcommand{\dfll}{\Longrightarrow}
\newcommand{\odfl}[1]{\overset{\displaystyle #1}{\dfll}}

\newcommand{\qfl}{\xymatrix@1@C=10pt{\ar@4 [r] &}}

\renewcommand{\tilde}[1]{\widetilde{#1}}

\renewcommand{\phi}{\varphi}
\renewcommand{\epsilon}{\varepsilon}

\newcommand{\Nb}{\mathbb{N}}
\newcommand{\Zb}{\mathbb{Z}}

\newcommand{\Fr}{\EuScript{F}}

\newcommand{\Ir}{\EuScript{I}}

\newcommand{\Lr}{\EuScript{L}}

\renewcommand{\Pr}{\EuScript{P}}

\newcommand{\Sr}{\EuScript{S}}
\newcommand{\Rr}{\EuScript{R}}
\newcommand{\Tr}{\EuScript{T}}

\newcommand{\ifthen}[2]{\ifthenelse{#1}{#2}{}}

\definecolor{cyan}{RGB}{175,238,238}

\renewcommand{\leq}{\leqslant}
\renewcommand{\geq}{\geqslant}

\newcommand{\Colo}[1]{\mathrm{Col}_{#1}}

\newcommand{\Scol}[1]{\mathrm{Scol}_{#1}}
\newcommand{\Young}[1]{\mathrm{Yt}_{#1}}
\newcommand{\Skew}[1]{\mathrm{Sk}_{#1}}
\newcommand{\dSkew}[1]{\mathrm{dSk}_{#1}}

\newcommand{\Scolc}[1]{\mathrm{Scol}^{\leq}_{#1}}

\def\glueY{\mathcal{Y}}

\newcommand{\Yrow}[1]{\mathbb{Y}^{r}_{#1}}

\newcommand{\dSKrow}[1]{d\mathbb{S}^{r}_{#1}}
\newcommand{\dSKcol}[1]{d\mathbb{S}^{c}_{#1}}

\newcommand{\Ycol}[1]{\mathbb{Y}^{c}_{#1}}

\def\rect{\pi_{tq}} % rectification d'un skew tableau
\newcommand{\Pl}{\mathbf{P}}

\def\Sr{\mathcal{S}}
\def\Fr{\mathcal{F}}
\def\Rr{\mathcal{R}}
\def\Lr{\mathcal{L}}
\def\Pr{\mathcal{P}}
\def\Ir{\mathcal{I}}
\def\Tr{\mathcal{T}}

\makeatletter
\def\blfootnote{\xdef\@thefnmark{}\@footnotetext}
\makeatother

\newcommand{\insl}[1]{\rightsquigarrow_{#1}}
\newcommand{\insr}[1]{\;\raisebox{0.1em}{{\rotatebox[origin=c]{180}{$\rightsquigarrow$}}}_{#1}\;}

\definecolor{Red}{rgb}{0.96,0.17,0.20}
\definecolor{RedD}{rgb}{0.57, 0.0, 0.04}
\definecolor{Green}{rgb}{0.0,1.0,0.0}
\definecolor{GreenL}{rgb}{0.0,1.0,0.0} %Light Green
\definecolor{GreenD}{rgb}{0.64,0.76,0.68} %Dark Green
\definecolor{Yellow}{rgb}{1.0,1.0,0.19}
\definecolor{YellowD}{rgb}{0.93, 0.84, 0.25}
\definecolor{Blue}{rgb}{0.0,1.0,1.0}
\definecolor{BlueD}{rgb}{0.63,0.79,0.95} % Dark Blue
\definecolor{vert}{rgb}{0,0.45,0}
\definecolor{bazaar}{rgb}{0.6, 0.47, 0.48}
\definecolor{bronze}{rgb}{0.8, 0.5, 0.2}
\definecolor{darkspringgreen}{rgb}{0.09, 0.45, 0.27}

\newcommand{\avoir}[1]{}%{{\color{MyGray}#1}}

\DeclareMathOperator{\Srs}{\mathcal{R}}

\DeclareMathOperator{\nf}{Nf}
\DeclareMathOperator{\fac}{Fac}

\def\points{\rotatebox{90}{\hspace{-1pt}...}}

\begin{document}
\thispagestyle{empty}

\begin{center}

\begin{doublespace}
\begin{huge}
{\scshape String of columns rewriting}

{\scshape and confluence of the jeu de taquin}
\end{huge}

\bigskip
\hrule height 1.5pt 
\bigskip

\begin{Large}
{\scshape Nohra Hage \qquad Philippe Malbos}
\end{Large}
\end{doublespace}

\bigskip

\begin{small}\begin{minipage}{14cm}
\noindent\textbf{Abstract -- } 
Schützenberger's jeu de taquin is an algorithm on the structure of tableaux, which transforms a skew tableau into a Young one by local transformation rules on the columns of the tableaux. 
This algorithm defines an equivalence relation on tableaux compatible with the plactic congruence, and gives a proof of the Littlewood--Richardson rule on Schur polynomials.  In this article, we introduce the notion of string of columns rewriting system as mechanism of transformations of glued sequences of columns. We describe the execution of the jeu de taquin algorithm as rewriting paths of a string of columns rewriting. We deduce algebraic properties on the plactic congruence and we relate the jeu de taquin to insertion algorithms on tableaux.

\medskip


\medskip

\smallskip\noindent\textbf{M.S.C. 2010 -- Primary:} 20M05. \textbf{Secondary:} 05E99, 68Q42, 20M35.
\end{minipage}\end{small}

%%%
%\vspace{1cm}

%%% Table des matières
%\begin{small}\begin{minipage}{12cm}
%\renewcommand{\contentsname}{}
%\setcounter{tocdepth}{2}
%\tableofcontents
%\end{minipage}
%\end{small}
\end{center}

%\clearpage
\medskip

%\tikzset{every tree node/.style={minimum width=1em,draw,circle},
%         blank/.style={draw=none},
%         edge from parent/.style=
%         {draw,edge from parent path={(\tikzparentnode) -- (\tikzchildnode)}},
%         level distance=0.8cm}

\section{Introduction}

Schützenberger introduced the \emph{jeu de taquin} as an algorithm on the structure of Young tableaux to prove  the \emph{Littlewood--Richardson rule} on the multiplicity of a Schur polynomial in a product of Schur polynomials, namely the multiplicity of an irreducible representation of the general Lie algebra in a tensor product of two irreducible representations,~\cite{Schutzenberger77}.
The jeu de taquin has later found many applications in algebraic combinatorics and probabilistic combinatorics~\cite{Fulton97,Leeuwen01, RomikSniady15}, and many similar algorithms were also introduced on other structures of tableaux,~\cite{Lecouvey02, Lecouvey03, ThomasYoung09, Tewari2015, RomikSniady15, Hage2021Super}.

A \emph{Young tableau} is  a collection of boxes in left-justified rows filled with elements of the totally ordered alphabet~$[n]:=\{1< \cdots<n\}$, where the entries weakly increase along each row and strictly increase down each column. A \emph{skew tableau} is obtained by eliminating boxes from the rows of a Young tableau starting from top to bottom and from left to right. The eliminated boxes located above and to the left of two non-empty boxes are called \emph{inner corners} of the skew tableau. We read tableaux column-wise from left to right and from bottom to top: the following tableaux
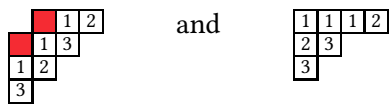
\[
\scalebox{1}{
\ytableausetup{smalltableaux}
\begin{ytableau}
\none & *(Red) & 1 & 2 \\
*(Red) & 1 & 3 \\
1 & 2 \\
3
\end{ytableau}
\qquad
\text{ and }
\qquad
\begin{ytableau}
1 & 1 & 1 & 2 \\
 2  &3\\
 3 \\
\end{ytableau}}
\]
are respectively skew tableau and Young tableau whose readings are $3121312$ and $3213112$, and where the empty red boxes denote the inner corners. 
The jeu taquin consists in applying  successively \emph{forward sliding operations} on a skew tableau that move an inner corner into an outer position by keeping the rows weakly increasing and the columns strictly increasing, until no more inner corners remain in the initial skew tableau, as follows
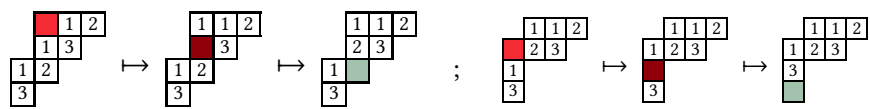
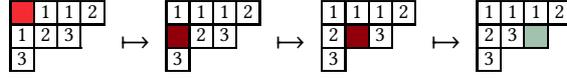
\[
\scalebox{1}{
\raisebox{0.45cm}{
\ytableausetup{smalltableaux}
\begin{ytableau}
\none & *(Red) & 1 & 2 \\
\none & 1 & 3 \\
1 & 2 \\
3
\end{ytableau}
\raisebox{-0.5cm}{$\;\mapsto\;$}
\begin{ytableau}
\none & 1 & 1 & 2 \\
\none &*(RedD)  &3 \\
1 & 2 \\
3
\end{ytableau}
\raisebox{-0.5cm}{$\;\mapsto\;$}
\begin{ytableau}
\none & 1 & 1 & 2 \\
\none &2  &3 \\
1 & *(GreenD) \\
3
\end{ytableau}
}}
\raisebox{-0.1cm}{$\quad ;\quad$}
\scalebox{0.9}{
\raisebox{0.45cm}{
\begin{ytableau}
\none & 1 & 1 & 2 \\
*(Red) &2  &3 \\
1 & \none \\
3
\end{ytableau}
\raisebox{-0.5cm}{$\;\mapsto\;$}
\begin{ytableau}
\none & 1 & 1 & 2 \\
1 &2  &3 \\
 *(RedD)& \none \\
3
\end{ytableau}
\raisebox{-0.5cm}{$\;\mapsto\;$}
\begin{ytableau}
\none & 1 & 1 & 2 \\
1 &2  &3 \\
 3& \none \\
*(GreenD)
\end{ytableau}
}}
\]
\[
\scalebox{1}{
\raisebox{0.45cm}{
\begin{ytableau}
*(Red) & 1 & 1 & 2 \\
1 &2  &3 \\
 3& \none \\
\end{ytableau}
\raisebox{-0.35cm}{$\;\mapsto\;$}
 \begin{ytableau}
1 & 1 & 1 & 2 \\
*(RedD) &2  &3 \\
 3& \none \\
\end{ytableau}
\raisebox{-0.35cm}{$\;\mapsto\;$}
 \begin{ytableau}
1 & 1 & 1 & 2 \\
2 &*(RedD)  &3 \\
 3& \none \\
\end{ytableau}
\raisebox{-0.35cm}{$\;\mapsto\;$}
\begin{ytableau}
1 & 1 & 1 & 2 \\
 2  &3& *(GreenD)\\
 3& \none \\
\end{ytableau}}}
\]

Schützenberger proved remarkable properties of the jeu de taquin on skew tableaux,~\cite{Schutzenberger77}. He proved that the \emph{rectification}  of a skew tableau by the jeu de taquin is a Young tableau whose reading is equivalent to the reading of the initial skew tableau with respect the \emph{plactic congruence} relation generated by the following \emph{Knuth relations},~\cite{Knuth70}:
\[
zxy = xzy, \; \text{ for }\; 1\leq x\leq y < z \leq n,
\;\text{and}\;\;\; 
yzx = yxz, \; \text{ for }\; 1\leq x<y\leq z\leq n.
\]
This congruence defines  the \emph{plactic monoid of type~A},~\cite{LascouxSchutsenberger81}, which emerged from the works of Schensted~\cite{Schensted61} and Knuth~\cite{Knuth70} on the combinatorial study of Young tableaux. Plactic monoids  have found several applications in algebraic combinatorics, representation theory, probabilistic combinatorics and rewriting theory,~\cite{Lothaire02,Fulton97,OConnell03, HageMalbos17, Hage15,CainGrayMalheiro19}.
Schützenberger proved  that the resulting Young tableau does not depend on the order in which we choose inner corners in the forward slidings. This is the \emph{confluence property} of the jeu de taquin. His proof follows the \emph{cross-section property} of Young tableaux with respect the plactic congruence, proved by Knuth in~\cite{Knuth70}, namely two words on~$[n]$ are plactic congruent if and only if they lead to the same Young tableau after applying Schensted's insertion algorithm,~\cite{Schensted61}. Explicitly, if there are two sequences of sliding operations that transform a tableau~$T$ into two different tableaux $T_1$ and $T_2$, then we continue applying sliding operations until we reach normal forms $\tilde T_1$  and $\tilde T_2$, that is tableaux without inner corners: 
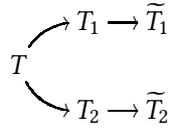
\[
\raisebox{0.6cm}{
\xymatrix @!C @C=1em @R=0.1em {
&  T_1
	\ar [r] ^{}
& {\tilde T_1}
\\
T
	\ar@/^/ [ur]^{}
	\ar@/_/[dr]_{}
\\
&  T_2
\ar [r] _{}
& {\tilde T_2}
}}
\]
Since~$\tilde T_1$ and $\tilde T_2$ are two Young tableaux such that their readings are plactic congruent, following the cross-section property, we deduce that $\tilde T_1=\tilde T_2$.

In this article, we introduce a machinery to prove by using a rewriting approach the confluence of the jeu de taquin. We define the sliding operations as rewriting rules on \emph{strings of columns}, that is strings composed by glued sequences of columns, with a gluing map that describes the relative positions of columns. This combinatorial structure generalizes many structures of tableaux such as skew tableaux,~\cite{Schutzenberger77}, Young tableaux of type $A$,~\cite{Young28}, Young tableaux of type $B$, $C$, $D$ and $G_2$,~\cite{KashiwaraNakashima94}, quasi-ribbon tableaux,~\cite{Novelli00}, and patience sorting structures~\cite{AldousPersi99}.
In Subsection~\ref{SS:Stringofcolumnsrewriting} we define a \emph{string of columns rewriting system} as a binary relation on the set of strings of columns, whose \emph{rules} are applied with respect to right and left positions. That is, a set of rules of the following  form
\[
 |_{p_s}\! u|_{q_s}\dfl |_{p_t}v|_{q_t}, 
\]
where $u,v$ are strings of columns and $p_s,q_s,p_t,q_t$ are positions in $\Zb$.
In Subsection~\ref{SS:JeutaquinAsRewriting} we present the jeu de taquin by the rewriting system~$\Fr\Sr_n$, whose rules are of the following form:
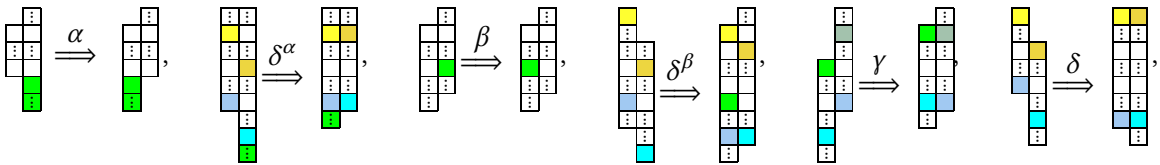
\[
\raisebox{0.5cm}{
\scalebox{0.7}{
\ytableausetup{mathmode, boxsize=0.8em}
\begin{ytableau}
\none & \points\\
\empty& \empty \\
\points &\points\\
\empty& \empty\\
\none & *(GreenL) \empty \\
\none &*(GreenL)\points
\end{ytableau}
}}
\odfl{\alpha}
\raisebox{0.5cm}{
\scalebox{0.7}{
\ytableausetup{mathmode, boxsize=0.8em}
\begin{ytableau}
\none & \points\\
\empty & \empty\\
\points &\points\\
\empty& \empty\\
*(GreenL)\empty&\none \\
*(GreenL)\points & \none
\end{ytableau}
}}, 
\quad
\raisebox{0.5cm}{
\scalebox{0.7}{
\ytableausetup{mathmode, boxsize=0.8em}
\begin{ytableau}
 \points&\points\\
 *(Yellow)\empty&\empty\\
 \points&\points\\
 \empty&*(YellowD)\empty  \\
\points &\points\\
*(BlueD) \empty& \empty \\
\none&\points  \\
\none& *(Blue)\empty \\
 \none&*(GreenL) \points
\end{ytableau}
}}
\raisebox{-0.3cm}{$\odfl{\delta^{\alpha}}$}
\raisebox{0.5cm}{
\scalebox{0.7}{
\ytableausetup{mathmode, boxsize=0.8em}
\begin{ytableau}
\points&\points  \\
*(Yellow) \empty&*(YellowD) \empty \\
\points&\points \\
\empty&\empty \\
\points&\points\\
*(BlueD) \empty& *(Blue) \empty \\
*(GreenL)\points&\none 
\end{ytableau}
}},
\quad
\raisebox{0.5cm}{
\scalebox{0.7}{
\ytableausetup{mathmode, boxsize=0.8em}
\begin{ytableau}
\none & \points\\
\empty & \empty\\
\points&\points\\
\empty& *(GreenL) \empty\\
\points&\points\\
\points & \none
\end{ytableau}
}}
\raisebox{-0.1cm}{$\odfl{\beta}$}
\raisebox{0.5cm}{
\scalebox{0.7}{
\ytableausetup{mathmode, boxsize=0.8em}
\begin{ytableau}
\none & \points\\
\empty & \empty\\
\points&\points\\
*(GreenL)\empty&\empty\\
\points&\points\\
\points & \none
\end{ytableau}
}},
\quad
\raisebox{0.5cm}{
\scalebox{0.7}{
\ytableausetup{mathmode, boxsize=0.8em}
\begin{ytableau}
*(Yellow)\empty &\none\\
\points&\none\\
\points&\points\\
\empty&*(YellowD)\empty \\
\points&\points\\
*(BlueD)\empty& \\
\points &\points\\
\none&\points\\
\none&*(Blue)\empty 
\end{ytableau}
}}
\raisebox{-0.5cm}{$\odfl{\delta^{\beta}}$}
\raisebox{0.5cm}{
\scalebox{0.7}{
\ytableausetup{mathmode, boxsize=0.8em}
\begin{ytableau}
\none&\points\\
*(Yellow)  \empty&\empty\\
\empty&*(YellowD)\empty \\
\points&\points  \\
\empty& \empty \\
*(GreenL) \empty&\empty \\
\points&\points\\
*(BlueD) \empty & *(Blue) \empty\\
\points&\none\\
\empty
\end{ytableau}
}},
\quad
\raisebox{0.5cm}{
\scalebox{0.7}{
\ytableausetup{mathmode, boxsize=0.8em}
\begin{ytableau}
\none & \points\\
\none & *(GreenD)\empty\\
\none & \points\\
*(GreenL)\empty & \empty \\
\points &\points\\
 \empty& *(BlueD) \empty\\
\points & \none\\
*(Blue)\empty & \none\\
\points& \none
\end{ytableau}
}}
\raisebox{-0.4cm}{$\odfl{\gamma}$}
\raisebox{0.5cm}{
\scalebox{0.7}{
\ytableausetup{mathmode, boxsize=0.8em}
\begin{ytableau}
\none & \points\\
*(GreenL) \empty &*(GreenD) \empty\\
\points &\points\\
\empty & \empty\\
\points&\points\\
*(Blue) \empty & *(BlueD) \empty\\
\points & \none
\end{ytableau}}},
\quad
\raisebox{0.5cm}{
\scalebox{0.7}{
\ytableausetup{mathmode, boxsize=0.8em}
\begin{ytableau}
*(Yellow)  \empty&\none\\
\points&\none\\
\empty&*(YellowD)\empty \\
\points &\points\\
*(BlueD) \empty& \empty \\
\none&\points  \\
 \none&*(Blue)\empty\\
 \none&\points
\end{ytableau}
}}
\raisebox{-0.4cm}{$\odfl{\delta}$}
\raisebox{0.5cm}{
\scalebox{0.7}{
\ytableausetup{mathmode, boxsize=0.8em}
\begin{ytableau}
 *(Yellow) \empty&*(YellowD)\empty\\
\points&\points\\
 \empty&\empty\\
\points &\points\\
 \empty&\empty \\
\points &\points\\
*(BlueD) \empty& *(Blue)\empty \\
\none&\points 
\end{ytableau}
}},
\]
see Subsection~\ref{SSS:RulesJeuDetaquin} for detailed positions of the columns.
For instance, the rectification of the above skew tableau is computed with the following reductions: 
\[
\raisebox{0.45cm}{
\ytableausetup{smalltableaux}
\begin{ytableau}
\none & \none & *(GreenD)  1 & 2 \\
\none &*(GreenL) 1 & *(BlueD)  3 \\
1 & *(Blue) 2 \\
3
\end{ytableau}
\raisebox{-0.5cm}{$\;\odfl{\gamma}\;$}
\begin{ytableau}
\none &  1 & 1 & 2 \\
\none &*(GreenD) 2  &3 \\
*(GreenL)1 \\
3
\end{ytableau}
\raisebox{-0.5cm}{$\;\odfl{\gamma}\;$}
\begin{ytableau}
\none & 1& 1 & 2 \\
 1 &*(GreenL) 2& 3 \\
  3 
\end{ytableau}
\raisebox{-0.5cm}{$\;\odfl{\beta}\;$}
\begin{ytableau}
1 & 1& 1 & 2 \\
 2 & \none & *(GreenL)3 \\
3 
\end{ytableau}
\raisebox{-0.5cm}{$\;\odfl{\alpha}\;$}
\begin{ytableau}
1 & 1& 1 & 2 \\
2 & 3 \\
3 
\end{ytableau}
}
\]

The main result of this article, Theorem~\ref{T:RewritingPropertiesYoung}, states that

\begin{quote}
{\bf i) } \emph{The rewriting system $\Fr\Sr_n$ is confluent and terminating.}

{\bf ii)} \emph{The normal forms with respect to~$\Fr\Sr_n$ are Young tableaux.}

{\bf iii)} \emph{Left and right Schensted's insertion algorithms coincide with the leftmost and rightmost normalization strategies of~$\Fr\Sr_n$.}

{\bf iv)} \emph{The rewriting system $\Fr\Sr_n$ computes the cross-section property.}
\end{quote}

The second result of this article, Theorem~\ref{T:MorphismRect}, proves the compatibility of the rewriting system~$\Fr\Sr_n$  with respect the plactic congruence.
Finally, from Theorems~\ref{T:RewritingPropertiesYoung} and~\ref{T:MorphismRect} we recover the cross-section property of Young tableaux with respect the plactic congruence and the commutation of right and left Schensted's insertion algorithms. In particular, we prove  that the rectification map defines a surjective morphism of monoids between the sets of diagonal skew tableaux and the set of Young tableaux equipped with the corresponding insertion products. 

\subsubsection*{Notations}
We will consider the totally ordered set $[n]:=\{1<\cdots<n\}$, for~$n\in\Zb_{>0}$, as ground alphabet. We denote by $[n]^\ast$ the free monoid of \emph{words over $[n]$}, whose empty word is denoted by~$\lambda$.  
We will denote by~$|w|$ the length of a word~$w$ over~$[n]$. 
We will denote by~$<_{\text{lex}}$ the lexicographic order on~$[n]^\ast$ induced by the order on~$[n]$, and by~$\preccurlyeq_{lex}$ and~$\preccurlyeq_{revlex}$ the lexicographic and the reverse lexicographic order respectively on tuples of natural numbers.

\section{String of columns rewriting}
\label{S:StringColumnsRewriting}

This section deals with two-dimensional strings defined by gluing columns by introducing the notion of \emph{string of columns} as a generalization of  the structure of Young tableaux.  
We  define in Subsection~\ref{SS:Stringofcolumnsrewriting} the notion of \emph{string of columns rewriting system} as a binary relation on the set of string of columns, whose rules are applied with respect to right and left positions and we show rewriting properties on theses rewriting systems.

\subsection{Strings of columns}
\label{SS:StringsOfColumns}

\subsubsection{Columns}
A \emph{column} (\emph{over~$[n]$}) is a decreasing string $c^k\ldots c^1$ over $[n]$, \emph{i.e.}, with $c^{i+1} > c^i$, for~$1\leq i < k$. It is represented by a collection of boxes in left-justified rows, filled with elements of~$[n]$, whose each row contains only one box, and the entries strictly increase down:
\[
c \: = \:
\scalebox{0.9}{
\raisebox{0.8cm}{
\ytableausetup
 {mathmode, boxsize=1.2em}
\begin{ytableau}
c^1 \\
c^2 \\
\none[\points]\\
c^k
\end{ytableau}
}}
\]
where $1\leq c^1 < \ldots < c^k \leq n$, and is also denoted by $(c^1;\ldots;c^k)$.
Denote~$|c|:= k$ the \emph{length} of $c$. We denote by~$\Colo{n}$ the set of columns over~$[n]$.
A column of length $0$ is the \emph{empty column} denoted by~$\lambda_c$. 

\subsubsection{String of columns}
Two columns $c_1,c_2$ in $\Colo{n}$ can be \emph{glued at position}~$p$ in $\mathbb{Z}$ as follows:
\[
c_1|_pc_2 
\:=\: 
\raisebox{1.35cm}{
\scalebox{0.9}{
\ytableausetup
 {mathmode, boxsize=1.5em}
\begin{ytableau}
\none & c^1_2\\
\none & \none[\points]\\
\none & c^i_2\\
c^1_1 & c^{i+1}_2\\
c^2_1 & \none[\points]\\
\none[\points]& \none[\points]\\
c^p_1 & c^{i+p}_2\\
\none[\points] & \none\\
c^k_1 & \none
\end{ytableau}
}}
\]
For $1\leq j \leq p$, we say that $(c_1^j,c_2^{i+j})$ is a \emph{full row of length $2$} in $c_1|_pc_2$. A pair $(\!\scalebox{0.35}{
\raisebox{-0.1cm}{
\ytableausetup{mathmode, boxsize=1.8em}
\begin{ytableau}
\,
\end{ytableau}
}}\,,c_2^j)$, for $1\leq j \leq i$, and a pair $(c_1^j,\scalebox{0.4}{
\raisebox{-0.1cm}{
\ytableausetup
 {mathmode, boxsize=1.8em}
\begin{ytableau}
\,
\end{ytableau}
}}\,)$, for $p+1\leq j \leq k$, is called a \emph{row of length $2$}, where $\!\!\scalebox{0.4}{
\raisebox{-0.1cm}{
\ytableausetup
 {mathmode, boxsize=1.8em}
\begin{ytableau}
\,
\end{ytableau}
}}$ denotes the \emph{empty box}.
A \emph{string of columns} is a sequence of glued columns:
\begin{eqn}{equation}
\label{E:StringOfColumns}
w \: = \: c_1|_{p_1} c_2 |_{p_2} \ldots c_{m} |_{p_{m}} c_{m+1}.
\end{eqn}
The sequence~$(p_1,p_2,\ldots,p_m)$ in $\Zb^m$ is called the \emph{gluing sequence} of~$w$. 
Gluing sequences can be defined in a consistent way by considering a \emph{gluing map} $g : \Colo{n} \times \Colo{n} \fl \Zb$ that associates to columns $c$ and~$c'$, a gluing position $g(c,c')$. Given a gluing map $g$, we define the set of strings of columns constructed with respect to $g$ as the set of string of columns of the form~(\ref{E:StringOfColumns}), where for any $1\leq i \leq m$, $p_i = g(c_i,c_{i+1})$.
The \emph{total length of $w$} is the tuple $tl(w)=(|c_1|,\ldots,|c_{m+1}|)\in~\Nb^{m+1}$. We will denote by~$||w||$ the number of columns of~$w$. 

For~$3\leq k\leq m+1$, a \emph{connected row of length $k$} is a sequence~$(c_{i_1}^{j_1},\ldots, c_{i_k}^{j_k})$ such that~$(c_{i_l}^{j_l},c_{i_{l+1}}^{j_{l+1}})$ is a full row of length~$2$ in $c_{i_l}|_{p_{i_l}}c_{i_{l+1}}$, for $1\leq l\leq k$. 
 A connected row~$(c_{i_1}^{j_1},\ldots, c_{i_k}^{j_k})$ is \emph{increasing} if~$c_{i_1}^{j_1}\leq\ldots\leq c_{i_k}^{j_k}$. 
We call a \emph{row} (\emph{over $[n]$}) a string of columns whose gluing sequence is constant equal to~$1$ and columns are of length~$1$.
A string of columns is \emph{row connected} (resp. \emph{row increasing}) if all its rows are connected (resp. increasing).

\subsubsection{Monoids of string of columns}
We will denote by $\Scol{n}$ the set of strings of columns over~$[n]$ and by~$\Scolc{n}$ the set of row connected and row increasing string of columns over~$[n]$.

Given a fixed gluing map $g$, we define a concatenation operation with respect to $g$ by the map $\cdot |_g \cdot : \Scol{n} \times \Scol{n} \fl \Scol{n}$, by setting
\[
(c_1|_{p_1} \ldots |_{p_{m}} c_{m+1} )
\;|_g\; 
(c_1'|_{q_1} \ldots |_{q_{n}} c'_{n+1})
=
c_1|_{p_1} \ldots |_{p_{m}} c_{m+1} |_{g(c_{m+1},c_1')}
c_1'|_{q_1} \ldots |_{q_{n}} c_{n+1}'.
\]
The operation $|_g$ is associative and unitary, where the identity is the \emph{empty string of columns} denoted by~$\lambda_c$.
We denote by $\Scol{n}^g$ the set of string of columns in $\Scol{n}$ whose gluing sequence is given by the gluing map $g$. In other words, $\Scol{n}^g$ is the free monoid on $\Colo{n}$ with respect the product $|_g$.

\subsubsection{The four corner readings}
\label{SSS:ReadingMaps}
The \emph{south-west reading} is the map $R_{SW}: \Scol{n} \fl [n]^\ast$ that reads a string of columns, column-wise  from left to right and from bottom to top. There are three other corner readings $R_{NW}$, $R_{NE}$, $R_{SW}$ and $R_{SE}$ defined in a similar way and that read a string of columns, column by column, with respect right or left and top or bottom directions.

Define the map~$\fac:[n]^\ast\to [n]^\ast$ sending a word $w$ into the factorization~$w=w_1\ldots w_k$, where each~$w_i$, for $i=1,\ldots,k$, is a maximal strictly decreasing sequence, that is, the $R_{SW}$-reading of a column in~$\Colo{n}$. 
For a fixed gluing map~$g\in\mathbb{Z}^n$,  consider the map
\begin{eqn}{equation}
[\;]_g:[n]^\ast \to \Scol{n}
\end{eqn}
that transforms each word~$w$ in~$[n]^\ast$ into a string of columns~$(c_1,\ldots,c_k)$ where each column~$c_i$ is filled by the elements of~$w_i$ in~$\fac(w)$ from bottom to top, for~$i=1,\ldots,k$, with respect the gluing map~$g$.

\subsubsection{Properties of strings of columns}
A row connected string of columns $w$ as in (\ref{E:StringOfColumns}) is called \begin{enumerate}[{\bf i)}]
\item \emph{left-justified} (resp. \emph{right-justified}) if $|c_i|\geq |c_{i+1}|$  and $|c_{i+1}|\leq p_i\leq |c_{i}|$ (resp.~$|c_{i+1}|\geq |c_{i}|$ and~$|c_{i}|\leq p_i\leq |c_{i+1}|$),  for all $1\leq i\leq m$.
\item \emph{top-justified}, (resp. \emph{bottom-justified}) if
$p_i=|c_{i+1}|$ (resp.\; $p_i=|c_i|$), for all $1\leq i\leq m$.
\item \emph{decreasing} (resp. \emph{increasing}) if its gluing sequence is decreasing (resp. increasing).
\end{enumerate}

\subsubsection{Example: skew tableaux}
\label{SSS:ExampleSkewTableaux}
A \emph{skew tableau} with $m+1$ columns is a string of columns $ c_1|_{p_1} \ldots |_{p_{m}} c_{m+1}$ in $\Scolc{n}$,  whose gluing sequence satisfies 
$p_k \leq |c_{k+1}|$, for all $1\leq k \leq m$.
A \emph{diagonal skew tableau} is a skew tableau $c_1|_{p_1} \ldots |_{p_{m}} c_{m+1}$ whose gluing sequence satisfies $p_k=1$, for all~$1\leq k \leq m$. We denote by~$s$ the gluing map for diagonal skew tableaux.
We will denote by $\Skew{n}$ (resp. $\dSkew{n}$) the set of skew (resp.~diagonal skew) tableaux over $[n]$.
Any string $u$ over $[n]$ is the $R_{SW}$-reading of a unique diagonal skew tableau, thus the map~$R_{SW}$ defines a bijection from $\dSkew{n}$ to~$[n]^\ast$. 
An \emph{inner corner} in a skew tableau~$w$ is an empty box located above and to the left of two non-empty boxes. An \emph{outer corner} in~$w$ is an empty box located to the end of a row or at the bottom of a column. 

We define the \emph{top} (resp. \emph{bottom}) \emph{concatenation} of an element $x$ in a column $c=(c^1;\ldots;c^k)$ as the skew tableau defined by
\[
c \insr{a} x
\: = \:
\begin{cases}
c |_1 x &\mbox{if } x\geq c^1,
\\
(x;c^1;\ldots;c^k) &\mbox{else.}
\end{cases}
\qquad
\mbox{\big(resp.\;}
x \insl{a} c
\: = \:
\begin{cases}
(c^1;\ldots;c^k;x) &\mbox{if } x> c^k,
\\
x|_1 c &\mbox{else.}
\end{cases}
\mbox{\;\big).}
\]

\noindent We extend these concatenations into insertion maps on skew tableaux, defined for $x\in[n]$ and $w=c_1|_{p_1}\ldots |_{p_m} c_m$ in $\Skew{n}$ by setting
\begin{eqn}{equation}
\label{E:LeftRightConcatenationSkew}
w \insr{I_r^a} x =c_1|_{p_1}\ldots |_{p_m} (c_m \insr{a} x),
\quad
\text{(resp.\;}
x \insl{I_l^a} w = (x \insl{a} c_1)|_{p_1}\ldots |_{p_m} c_m
\; \text{).}
\nonumber
\end{eqn}

For any word $w=x_1\ldots x_k$, denote by~$C_{\dSKrow{n}}(w)$ (resp.~$C_{\dSKcol{n}}(w)$) the diagonal skew tableau obtained from~$w$ by inserting its letters iteratively from left to right (resp. right to left) using the right (resp. left) insertion starting from the empty tableau:
\[
\begin{array}{rl}
C_{\dSKrow{n}}(w)
\;
:=&
\;
(\emptyset \insr{I_{r}^{a}} w)
\;
=
\;((\ldots(\emptyset \insr{I_{r}^{a}} x_1)  \insr{I_{r}^{a}} \ldots )\insr{I_{r}^{a}} x_k),\\
\big(\text{resp. }
C_{\dSKcol{n}}(w)
\;
:=&
\;
(w\insl{I_{l}^{a}} \emptyset )
\;
=
(x_1 \insl{I_{l}^{a}} ( \ldots\insl{I_{l}^{a}}  (x_k\insl{I_{l}^{a}} \emptyset)\ldots))\big).
\end{array}
\]
Define now an internal product~$\star_{I_{r}^{a}}$ (resp.~$\star_{I_{l}^{a}}$) on~$\dSkew{n}$ by setting 
\begin{eqn}{equation}
\label{Eq:StructureMonoidProduct}
t \star_{I_{r}^{a}} t' :=  (t\insr{I_{r}^{a}} R_{SW}(t')),
\qquad
\big(\text{resp. }  t \star_{I_{l}^{a}} t' :=  (R_{SW}(t')\insl{I_{l}^{a}}t)\big)
\end{eqn}
for all $t,t'$ in $\dSkew{n}$.
By definition the relations $t\star_{I_{r}^{a}} \emptyset = t$ (resp. $t\star_{I_{l}^{a}} \emptyset = t$) and $\emptyset \star_{I_{r}^{a}} t = t$ (resp.~$\emptyset \star_{I_{l}^{a}} t = t$) hold, showing that the product~$\star_{I_{r}^{a}}$ (resp.~$\star_{I_{l}^{a}}$) is unitary with respect to~$\emptyset$.

The top (resp. bottom) concatenation on a diagonal skew tableau~$w$ acts only on the last (resp. first) column of~$w$ and do not change the others columns. As a consequence, for all~$x,y\in [n]$, we have the following \emph{commutation property}: 
\begin{eqn}{equation}
\label{E:CommutationAppend}
y \insl{I_l^a} (w \insr{I_r^a} x)
\: = \:
(y \insl{I_l^a} w) \insr{I_r^a} x.
\end{eqn}
Hence, we deduce that the insertion products~$\star_{I_{r}^{a}}$  and~$\star_{I_{l}^{a}}$ are associative.

\subsubsection{Example: Young tableaux}
\label{SSS:YoungTableaux}
A \emph{Young tableau} with $m+1$ columns is a string of columns $ c_1|_{p_1} \ldots  |_{p_{m}} c_{m+1}$ in~$\Scolc{n}$ whose gluing sequence is decreasing and satisfies
\begin{eqn}{equation}
\label{E:gluingSequenceYoung}
p_k = |c_{k+1}|
\quad 
\text{for all $1\leq k \leq m$.}
\end{eqn}
We denote by~$\glueY$ the gluing map for Young tableaux, and by~$\Young{n}$ the set of Young tableaux over~$[n]$.

Given a row $r$ (resp. a column~$c$), we denote  by \textsc{RowInsert}$(r,x)$ (resp.~\textsc{ColumnInsert}$(c,x)$) the procedure that inserts an element $x$ in a row~$r$ (resp.~column~$c$) and returns a pair $(r',y)$ (resp.~$(c',y)$) made of the resulting row $r'$ (resp.~column $c'$) and the bumping element $y$ that can be empty, as follows. If $x$ is bigger or equal (resp. strictly bigger) than all the elements of~$r$ (resp.~$c$), then~$r'$ (resp.~$c'$) is obtained by adding~$x$ to the end (resp.~the bottom) of~$r$ (resp.~$c$) and~$y$ is empty. Otherwise, let~$y$ be the smallest element of~$r$ (resp.~$c$) such that~$x<y$ (resp.~$x\leq y$), then~$r'$ (resp.~$c'$) is obtained from~$r$ (resp.~$c$) by replacing~$y$ by~$x$.
The \emph{right} (resp.~\emph{left}) \emph{insertion algorithm} computes a tableau $(t \insr{S_r} x)$ (resp.~$(x \insl{S_l} t)$) as follows,\cite{Schensted61}: 

\setlength{\columnsep}{0.1cm}
\begin{multicols}{2}
\scalebox{0.9}{
\begin{algorithm}[H]
\DontPrintSemicolon

\noindent \textsc{RightInsertYT}$(t,x)$ 

\KwIn{A Young tableau $t$ and $x$ in $[n]$.}
\KwOut{The Young tableau $(t \insr{S_r} x)$.}

\BlankLine

$y := x$ ; $t' := \emptyset$ ;

\While{$y\neq \lambda$}{

$r:= t[1]$ ;

$t:= t/r$ ;

$(r',y):=\textsc{RowInsert}(r,y)$

$t' : = (t' ; r')$

}

\Return{$(t' ; t)$}

\BlankLine

\BlankLine

\caption{Schensted's right algorithm}
\label{A:RightSchensted}
\end{algorithm}
\quad
\begin{algorithm}[H]
\DontPrintSemicolon

\noindent \textsc{LeftInsertYT}$(t,x)$

\KwIn{A Young tableau $t$ and $x$ in $[n]$.}
\KwOut{The Young tableau $(x \insl{S_l} t)$.}

\BlankLine

$y := x$ ; $t' := \emptyset$ ;

\While{$y\neq \lambda$}{

$c:= t_{[1]}$ ;

$t:= t_{/c}$ ;

$(c',y):=\textsc{ColumnInsert}(c,y)$

$t' : = [t' ; c']$

}

\Return{$[t' ; t]$}

\BlankLine

\BlankLine

\caption{Schensted's left algorithm}
\label{A:LeftSchensted}
\end{algorithm}}
\end{multicols}

\noindent where $t[i]$ (resp.~$t_{[i]}$) denotes the $i$-th row (resp.~column) of the tableau $t$, and $t/t[1]$ (resp.~$ t_{/t_{[1]}}$) the Young tableau obtained from~$t$ by removing its first row (resp.~column), and where $(t;t')$ (resp.~$[t;t']$) denotes the Young tableau obtained by concatenating $t$ over (resp.~to the right of) a Young tableau $t'$ when the concatenation defines a Young tableau.

For any word $w=x_1\ldots x_k$, denote by~$C_{\Yrow{n}}(w)$ (resp.~$C_{\Ycol{n}}(w)$) the Young tableau obtained from~$w$ by inserting its letters iteratively from left to right (resp. right to left) using the right (resp. left) insertion starting from the empty tableau:
\[
\begin{array}{rl}
C_{\Yrow{n}}(w)
\;
:=&
\;
(\emptyset \insr{S_r} w)
\;
=
\;((\ldots(\emptyset \insr{S_r} x_1)  \insr{S_r} \ldots )\insr{S_r} x_k),\\
\big(\text{resp. }
C_{\Ycol{n}}(w)
\;
:=&
\;
(w\insl{S_l} \emptyset )
\;
=
(x_1 \insl{S_l} ( \ldots\insl{S_l}  (x_k\insl{S_l} \emptyset)\ldots))\big).
\end{array}
\]
Define now an internal product~$\star_{S_r}$ (resp.~$\star_{S_l}$) on~$\Young{n}$ by setting 
\begin{eqn}{equation}
\label{Eq:StructureMonoidProduct}
t \star_{S_r} t' :=  (t\insr{S_r} R_{SW}(t')),
\qquad
\big(\text{resp. }  t \star_{S_l} t' :=  (R_{SW}(t')\insl{S_l}t)\big)
\end{eqn}
for all $t,t'$ in $\Young{n}$.
By definition the relations $t\star_{S_r} \emptyset = t$ (resp. $t\star_{S_l} \emptyset = t$) and $\emptyset \star_{S_r} t = t$ (resp.~$\emptyset \star_{S_l} t = t$) hold, showing that the product~$\star_{S_r}$ (resp.~$\star_{S_l}$) is unitary with respect to~$\emptyset$.

Let~$c$ be a column  of length~$p$, the \emph{Sch\"{u}tzenberger involution} of~$c$, denoted by~$c^{\ast}$, is the column of length~$n-p$ obtained by taking the complement of the elements of~$c$. This involution is extended to string of columns by setting~$(c_1|\ldots |c_r)^\ast= c_r^\ast|\ldots |c_1^\ast$, for all~$c_1,\ldots, c_r$ in~$\Colo{n}$. If~$c_1|_\glueY\ldots |_\glueY c_r$ is a Young tableau,  then~$(c_1|_\glueY \ldots |_\glueY c_r)^{\ast} = c_r^{\ast}|_\glueY \ldots |_\glueY c_1^{\ast}$  is also a Young tableau. Moreover,  the following equality~$(c_1\star_{S_r} \ldots \star_{S_r} c_r)^\ast = (c_r^{\ast}\star_{S_r} \ldots\star_{S_r} c_1^{\ast})$ holds, for all~$c_1,\ldots, c_r$ in~$\Colo{n}$. In particular, for three columns~$c_i$,~$c_j$ and~$c_k$ in~$\Colo{n}$, we have~$(c_i\star_{S_r} c_j\star_{S_r} c_k)^{\ast} =(c_k^{\ast}\star_{S_r}c_j^{\ast}\star_{S_r} c_i^{\ast})$, see~{\cite[Remark~3.2.7]{HageMalbos17}}.

\subsection{String of columns rewriting}
\label{SS:Stringofcolumnsrewriting}

\subsubsection{Rewriting steps}
Let define $\Scol{n}^\Pr = \Zb \times \Scol{n} \times \Zb$, whose elements are triples $(p,u,q)$ where~$p,q$ are positions in $\Zb$ and $u$ is a string of columns, that we will denote by $ |_pu|_q$. 
We define a \emph{string of columns rewriting system}, called \emph{rewriting system} for short in the sequel, as a binary relation on $\Scol{n}^\Pr$. That is, a set of \emph{rules} of the form
\begin{eqn}{equation}
\label{E:Rule}
\alpha : |_{p_s}\! u|_{q_s}\dfl |_{p_t}v|_{q_t}, 
\end{eqn}
where $u,v$ are strings of columns in $\Scol{n}$ and $p_s,q_s,p_t,q_t$ are positions in $\Zb$.
The pair $(p_s,q_s)$ (resp.~$(p_t,q_t)$) is called the \emph{source} (resp. \emph{target }) \emph{positions}, and $u$ (resp. $v$) is called the \emph{string of columns source} (resp. \emph{target}) of the rule $\alpha$, denoted by $s(\alpha)$ (resp. $t(\alpha)$).

A string of columns $w$ is said to be \emph{reducible with respect to $\alpha$}, if there is a factorization $w = w_1|_{p_s}s(\alpha)|_{q_s}w_2$ in $\Scol{n}$. In that case, $w$ \emph{reduces into} $w'=w_1|_{p_t}t(\alpha)|_{q_t}w_2$. Such a \emph{reduction} is denoted by $w_1|_{p_s}\alpha|_{q_s} w_2$, or $\alpha$ if there is no possible confusion. Given a rewriting system $\Rr$, the set of all reductions defines a binary relation on $\Scol{n}$, called the \emph{$\Rr$-rewrite relation} that we will denote by $\dfl_\Rr$, or~$\dfl$ if there is no possible confusion. The elements of $\dfl_{\Rr}$ are called \emph{$\Rr$-rewriting steps}, and have the form
\begin{eqn}{equation}
\label{E:ReductionStep}
|_{p_1}w_1 |_{p_s} s(\alpha) |_{q_s} w_2 |_{p_2}\dfl |_{p_1}w_1 |_{p_t} t(\alpha) |_{q_t} w_2|_{p_2},
\end{eqn}
for all $\alpha$ in $\Rr$ and $w_1,w_2$ in~$\Scol{n}$.
In~(\ref{E:ReductionStep}) the data $|_{p_1}w_1 |_{p_s} - |_{q_s} w_2|_{p_2}$ is called the \emph{context} of the rule~$\alpha$. If we denote by $C$ this context, the reduction~(\ref{E:ReductionStep}) can be also denoted by $C[\alpha]$.

We denote by $\dfl_\Rr^\ast$ the reflexive and transitive closure of the relation $\dfl_\Rr$, whose elements are called \emph{$\Rr$-rewriting paths}.

\subsubsection{Rewriting properties}
A rewriting system $\Rr$ is \emph{terminating} if there is no infinite $\Rr$-rewriting path.
A \emph{local branching} (resp. \emph{branching}) of $\Rr$ is a pair $(\varphi,\psi)$ of $\Rr$-rewriting steps (resp. $\Rr$-rewriting paths) having the same source as depicted in the following reduction diagram:
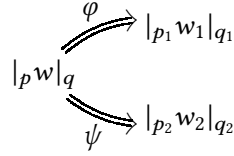
\[
\xymatrix @C=1.8em @R=0em{
& |_{p_1}w_1|_{q_1}
\\
|_pw|_q
  \ar@2@/^2ex/[ur] ^-{\varphi}
  \ar@2@/_2ex/[dr] _-{\psi}
&\\
& |_{p_2}w_2|_{q_2}
}
\]
Such a branching is \emph{confluent} if there exist $\Rr$-rewriting paths $\varphi'$ and $\psi'$ with a common target as follows:
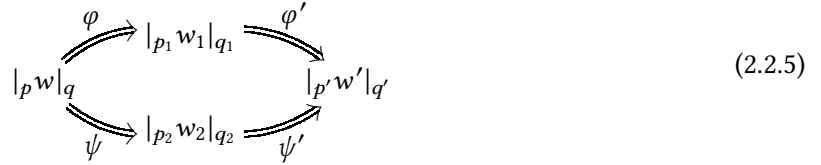
\begin{eqn}{equation}
\label{E:confluentBranching}
\raisebox{0.5cm}{
\xymatrix @C=1.8em @R=0em{
& 
|_{p_1}w_1 |_{q_1}
  \ar@2@/^2ex/[dr] ^-{\varphi'}
\\
|_p w |_q
  \ar@2@/^2ex/[ur] ^-{\varphi}
  \ar@2@/_2ex/[dr] _-{\psi}
&&
|_{p'} w' |_{q'}
\\
& 
|_{p_2} w_2 |_{q_2}
  \ar@2@/_2ex/[ur] _-{\psi'}
}}
\end{eqn}
We say that $\Rr$ is \emph{locally confluent} (resp. \emph{confluent}) if any local branching (resp. branching) of $\Rr$ is confluent, and that $\Rr$ is convergent if it is confluent and terminating.
A string of columns $w$ is in \emph{normal form with respect to $\Rr$}, if there is no rule that reduces $w$. 
When $\Rr$ is convergent, any string of columns~$w$ has a unique normal form, denoted by $\nf(w,\Rr)$.

\subsubsection{Critical branching}
A local branching of the form $(\varphi,\varphi)$ is called \emph{aspherical}.
A local branching~$(\varphi,\psi)$ is called \emph{orthogonal} if the source of $\varphi$ does not \emph{overlap} with the source of $\psi$, that is the source of the branching is of the form $|_{p_0}w_1|_{p_1}s(\varphi)|_{q_1}w_2|_{p_2}s(\psi)|_{q_2}w_3|_{p_3}$, with $w_1,w_2,w_3$ in $\Scol{n}$.
A local branching that is neither aspherical nor orthogonal is called \emph{overlapping}. 
There are three shapes of overlapping branchings $(\varphi,\psi)$, where
$s(\varphi)=|_{p_0}w_1|_{p_s}s(\alpha)|_{q_s}w_2|_{p_2}$, and $s(\psi)=|_{p_0'}w_1'|_{p'_s}s(\beta)|_{q'_s}w'_2|_{p_2'}$, with~$\alpha,\beta\in\Rr$, described by the following situations:
\begin{enumerate}[{\bf i)}]
\item (\emph{position overlapping}) $q_s=q'_s$,
\item (\emph{string overlapping}) $s(\alpha) = |_{p_s} u_\alpha |_{p'_s} v_\alpha |_{q_s}$ and $s(\beta) = |_{p'_s}v_\alpha |_{q_s}u_\beta|_{q'_s}$,
\item (\emph{inclusion}) $s(\beta) = |_{p'_s}u_\beta |_{p_s} s(\alpha) |_{q_s} v_{\beta} |_{q'_s}$.
\end{enumerate}

An overlapping branching that is minimal for the relation $\sqsubseteq$ on branchings generated by 
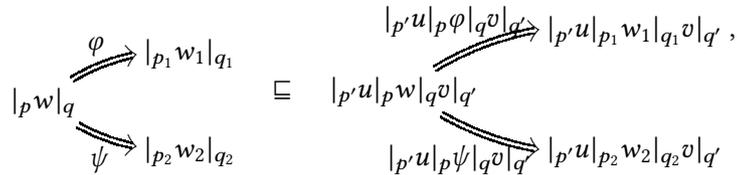
\[
\raisebox{-0.2cm}{\xymatrix @C=1.8em @R=0.1em{
& |_{p_1}w_1|_{q_1}
\\
|_p w|_q
  \ar@2@/^/[ur] ^-{\varphi}
  \ar@2@/_/[dr] _-{\psi}
&\\
& |_{p_2}w_2|_{q_2}
}}
\quad\raisebox{-0.8cm}{$\sqsubseteq$}\quad
\xymatrix @C=1.8em @R=0.5em{
&
|_{p'}u|_{p_1}w_1|_{q_1}v|_{q'}
\\
|_{p'}u|_pw|_qv|_{q'}
  \ar@2@/^2ex/[ur] ^-{|_{p'}u|_p\varphi |_qv|_{q'}}
  \ar@2@/_2ex/[dr] _-{|_{p'}u|_p\psi|_qv|_{q'}}
&\\
&
|_{p'}u|_{p_2}w_2|_{q_2}v|_{q'}
},
\]
for any branching $(\varphi,\psi)$ and context $u |_p - |_q v$ of reductions $\varphi,\psi$, is called a \emph{critical branching}.

\begin{lemma}
\label{L:Convergence}
A rewriting system $\Rr$ is locally confluent if and only if all its critical branchings are confluent. 
Moreover, if $\Rr$ is terminating with all its critical branchings are confluent, then it is confluent.
\end{lemma}

\begin{proof}
The first statement is the critical branching lemma.
Suppose that all the critical branchings of $\Rr$ are confluent and prove that any local branching of $\Rr$ is confluent. By definition, every aspherical branching is trivially confluent, and every orthogonal local branching is confluent.
Consider an overlapping but not minimal local branching $(\varphi,\psi)$, there exist factorizations $\varphi=C[\varphi']$ and $\psi=C[\psi']$, where $(\varphi',\psi')$ is a critical branching of $\Rr$. By hypothesis, this branching is confluent, and there are reductions paths $\phi'':t(\varphi') \fl w$ and $\psi'':t(\psi') \fl w$ that reduce targets of $\phi'$ and $\psi'$ to the same string of columns $w$.
It follows that the reductions paths~$C[\phi'']$ and~$C[\psi'']$ make the branching~$(\phi,\psi)$ confluent.

The second statement is an immediate consequence of Newman's lemma, \cite{Newman42}, that proves that any locally confluent terminating rewriting system is confluent.
\end{proof}

\subsubsection{Normalization strategies}
A \emph{reduction strategy} for a rewriting system $\Rr$ specifies a way to apply the rules in a deterministic way.
When $\Rr$ is normalizing, a \emph{normalization strategy} is a mapping $\sigma$ of every string of columns $|_pw|_q$ to a rewriting path $\sigma_{|_pw|_q}$ with source $|_pw|_q$ and target a  chosen normal form of~$|_pw|_q$ with respect to~$\Rr$. For a reduced rewriting system, we distinguish the leftmost reduction strategy and the rightmost one, according to the way we apply first the rewriting rule that reduces the leftmost or the rightmost string of columns. They are defined as follows. 
For every string of columns~$|_pw|_q$, the set of rewriting steps with source~$|_pw|_q$ can be ordered from left to right by setting~$\phi\prec \psi$, for rewriting steps 
$\phi = |_pw_1|_{p_s}\alpha|_{q_s}w_2|_q$
and
$\phi = |_pw_1'|_{p_s}\beta|_{q_s}w_2'|_q$
such that~$||w_1|| <||w_1'||$. If~$\Rr$ is finite, then the order~$\prec$ is total and the set of rewriting steps of source~$|_pw|_q$ is finite. Hence this set contains a smallest element~$\sigma_{|_pw|_q}$ and a greatest element~$\eta_{|_pw|_q}$, respectively called the \emph{leftmost} and the \emph{rightmost rewriting steps on~$|_pw|_q$}. If, moreover, the rewriting system terminates, the iteration of~$\sigma$ (resp.~$\eta$) yields a normalization strategy for~$\Rr$ called the \emph{leftmost} (resp.\ \emph{rightmost}) \emph{normalization strategy of~$\Rr$}: 
\[
\rho_\Rr^\top(|_pw|_q) = \sigma_{|_pw|_q} | \rho_\Rr^\top (t(\sigma_{|_pw|_q}))
\qquad
(\text{resp.\ }
\rho_\Rr^\perp(|_pw|_q) = \eta_{|_pw|_q} | \rho_\Rr^\perp(t(\eta_{|_pw|_q}))\;).
\]
The \emph{leftmost} (resp. \emph{rightmost}) \emph{rewriting path} on a string of columns $|_pw|_q$  is the rewriting path obtained by applying the leftmost (resp. rightmost) normalization strategy $\rho^\top_{\Rr}$ (resp. $\rho^\perp_{\Rr}$).
We refer the reader to~\cite{GuiraudMalbos12advances} and \cite{GuiraudMalbos18} for more details on rewriting normalization strategies.

\subsubsection{Top-left sliding order}
\label{SSS:TopLEftSlidingOrder}
A way to prove termination of a string of columns rewriting system~$\Rr$ is to consider a map~\mbox{$f: \Scol{n} \fl (X,\prec)$,} where $(X,\prec)$ is a well-ordered set satisfying, for all $w,w'\in \Scol{n}$,
\[
w \dfl_\Rr w' 
\quad
\text{implies}
\quad
f(w') \prec f(w).
\]
We will use the following termination order.
Let $w=u_{1}|_{p_1}\ldots|_{p_{m-1}} u_{m}$ be in $\Scol{n}$.
Denote by $h_{u_k}$ the number of empty boxes between the top box of the column~$u_k$ and the \emph{top position of $w$}, shown by the blue line in the following  picture 
\[
\scalebox{0.4}{
\begin{tikzpicture}[inner sep=0in,outer sep=0in]
\node (n) {\begin{varwidth}{5cm}{
\ytableausetup{mathmode, boxsize=1.9em}
\begin{ytableau}
\none&\none&\none&\empty&\none&\none\\
\none&\none&\none&\points&\none&\none\\
\none&\none&\empty&\empty&\none&\none \\
\points&\none[\ldots]&\points&\points&\none&\none[\ldots]&\points\\
\empty&\none[\ldots]&\empty&\empty&\none&\none[\ldots]&\empty\\
\empty&\none[\ldots]&\none&\empty&\empty&\none[\ldots]&\none \\
\points&\none&\none&\points&\points&\none 
\end{ytableau}}
\end{varwidth}
};
\draw[ultra thick,blue](-2.48,2.6)--(2.73,2.6);
\draw[ultra thick,Green](-2.48,0.38)--(-2.48,2.6);
\draw[ultra thick,Green](2.73,2.6)--(2.73,0.38);
\end{tikzpicture}
}
\]
Define the \emph{top deviation of $w$} as the sequence~$d^\top(w) = (h_{u_{1}},\ldots,h_{u_{m}}) \in \Nb^{m}$.
Denote by~$\preceq_{lex}$ the total order on~$\Scol{n}$ defined by $w\preceq_{lex} w'$ if and only if
\[
tl(w)\prec_{revlex} tl(w')\; \text{ or } \;\big(\,tl(w) = tl(w') \;\text{ and }\; d^\top(w) \prec_{lex} d^\top(w')\,\big).
\]
In order to prove the termination of top-left sliding operations presented in~\ref{SSS:RulesJeuDetaquin}, we define the total order~$\ll_{tl}$ on~$\Scol{n}$ by setting, for $w,w'$ in~$\Scol{n}$, $w\ll_{tl} w'$ if and only if
\[
||w||<||w'|| \;\text{ or }\;\big(||w|| = ||w'|| \text{ and } w \preceq_{lex} w'\big).
\]

\section{Convergence of the jeu de taquin}
\label{S:ConvergenceJeuTaquin}

In this section, we study the confluence of the jeu de taquin through a rewriting system defined by column slidings. We show that this rewriting system is convergent and we present the jeu de taquin as a surjective map from the set of diagonal skew tableaux to the set of Young tableaux using insertion. We recover properties relating the jeu de taquin to the plactic congruence and insertion algorithms on the structure of Young tableaux.

\subsection{Jeu de taquin}
\label{SS:JeuDetaquin}

\subsubsection{Plactic monoids}
Recall that the \emph{plactic monoid (of type $A$) of rank $n$}, introduced in~\cite{LascouxSchutsenberger81}, and denoted by~$\Pl_n$, is generated on $[n]$ and submitted to the following \emph{Knuth relations},~\cite{Knuth70}:
\begin{eqn}{equation}
\label{E:KnuthRelations}
zxy = xzy, \; \text{ for }\; 1\leq x\leq y < z \leq n,
\;\text{and}\;\;\; 
yzx = yxz, \; \text{ for }\; 1\leq x<y\leq z\leq n.\end{eqn}
Knuth in \cite{Knuth70} described the congruence~$\approx_{\Pl_n}$ generated by these relations in terms of Young tableaux and proved the cross-section property for the monoid $\Pl_n$.

\subsubsection{Forward sliding,~\cite{Schutzenberger77}}
A \emph{forward sliding} is a sequence of the following slidings:
\[
\ytableausetup{mathmode, smalltableaux}
\scalebox{1}{
\begin{ytableau}
 *(Red)&y\\
   x
\end{ytableau}}
 \;\leftrightarrow\;
\scalebox{1}{
 \begin{ytableau}
x& y\\
 *(Red)&\none
\end{ytableau}}
 \quad\text{ for } x\leq y,  
\qquad
\scalebox{1}{
\begin{ytableau}
 *(Red)&x\\
y
\end{ytableau}}
\;\leftrightarrow\;
\scalebox{1}{\begin{ytableau}
x& *(Red)\\
y&\none 
\end{ytableau}}
\quad\text{ for } x< y,  
\qquad
\scalebox{1}{\begin{ytableau}
 *(Red)&\empty\\
x
\end{ytableau}}
\;\leftrightarrow\;
\scalebox{1}{\begin{ytableau}
x&\empty\\
 *(Red)&\none 
\end{ytableau}}\; ,
\qquad
\scalebox{1}{\begin{ytableau}
 *(Red) &x\\
\empty&\none
\end{ytableau}}
\;\leftrightarrow\;
\scalebox{1}{\begin{ytableau}
x& *(Red)\\
\empty&\none 
\end{ytableau}}
\]
starting from a skew tableau and one of its inner corners, and moving the empty box until it becomes an outer corner.
The \emph{jeu de taquin} on a skew tableau $w$ consists in applying successively the forward slide algorithm starting from~$w$ until we get a string of columns without inner corners denoted $\rect(w)$, which is shown to be a Young tableau. In this way, the jeu de taquin defines a map
\[
\rect:\Skew{n}\to \Young{n},
\]
also called the \emph{rectification} of skew tableaux.
Sch\"{u}tzenberger proved in~\cite{Schutzenberger77} many properties of the jeu de taquin. These properties are also presented by Fulton in~\cite{Fulton97}, as follows. For any~$w\in~\Skew{n}$, the following conditions hold
\begin{enumerate}[{\bf i)}]
\item~{\cite[Proposition~2]{Fulton97}}. $R_{SW}(w) \approx_{\Pl_n} R_{SW}(\rect(w))$,
\item~{\cite[Corollary~1]{Fulton97}}. The rectification $\rect(w)$ is the unique Young tableau satisfying {\bf i)},
\item~{\cite[Claim~2]{Fulton97}}. The map $\rect$ does not depend on the order in which the inner corners are chosen.
\end{enumerate}

Note that condition {\bf ii)} is a consequence of the cross-section property for~$\Pl_n$ proved in~\cite{Knuth70} and condition {\bf i)}. Moreover, condition {\bf iii)} is a consequence of conditions {\bf i)} and {\bf ii)}.
In the rest of this section, we show that conditions {\bf i)} and {\bf ii)} are direct consequence of a confluence property of a rewriting system that computes the map~$\rect$, and without supposing the cross-section property for~$\Pl_n$ which will be also consequence of this confluence property.

\subsubsection{Example}
\label{Ex:Jeudetaquin}
The jeu de taquin on the following skew tableau $w$ starting with the inner corner~$\scalebox{0.7}{\ytableausetup{smalltableaux}
\begin{ytableau}
 *(Red)
\end{ytableau}}$ applies three occurrences of forward sliding, where 
$\scalebox{0.7}{\ytableausetup{smalltableaux}
\begin{ytableau}
 *(RedD)
\end{ytableau}}$ denotes the empty box, and
$\scalebox{0.7}{\ytableausetup{smalltableaux}
\begin{ytableau}
 *(GreenD)
\end{ytableau}}$ the outer corner:
\[
w\:=\: 
\raisebox{0.45cm}{
\ytableausetup{smalltableaux}
\begin{ytableau}
\none & *(Red) & 1 & 2 \\
\none & 1 & 3 \\
1 & 2 \\
3
\end{ytableau}
\raisebox{-0.5cm}{$\quad\mapsto\quad$}
\begin{ytableau}
\none & 1 & 1 & 2 \\
\none &*(RedD)  &3 \\
1 & 2 \\
3
\end{ytableau}
\raisebox{-0.5cm}{$\quad\mapsto\quad$}
\begin{ytableau}
\none & 1 & 1 & 2 \\
\none &2  &3 \\
1 & *(GreenD) \\
3
\end{ytableau}
}
\raisebox{-0.1cm}{$\qquad ;\qquad$}
\raisebox{0.45cm}{
\begin{ytableau}
\none & 1 & 1 & 2 \\
*(Red) &2  &3 \\
1 & \none \\
3
\end{ytableau}
\raisebox{-0.5cm}{$\quad\mapsto\quad$}
\begin{ytableau}
\none & 1 & 1 & 2 \\
1 &2  &3 \\
 *(RedD)& \none \\
3
\end{ytableau}
\raisebox{-0.5cm}{$\quad\mapsto\quad$}
\begin{ytableau}
\none & 1 & 1 & 2 \\
1 &2  &3 \\
 3& \none \\
*(GreenD)
\end{ytableau}
}
\]
\[
\begin{ytableau}
*(Red) & 1 & 1 & 2 \\
1 &2  &3 \\
 3& \none \\
\end{ytableau}
\raisebox{-0.35cm}{$\quad\mapsto\quad$}
 \begin{ytableau}
1 & 1 & 1 & 2 \\
*(RedD) &2  &3 \\
 3& \none \\
\end{ytableau}
\raisebox{-0.35cm}{$\quad\mapsto\quad$}
 \begin{ytableau}
1 & 1 & 1 & 2 \\
2 &*(RedD)  &3 \\
 3& \none \\
\end{ytableau}
\raisebox{-0.35cm}{$\quad\mapsto\quad$}
\begin{ytableau}
1 & 1 & 1 & 2 \\
 2  &3& *(GreenD)\\
 3& \none \\
\end{ytableau}
\raisebox{-0.35cm}{$\:=\:\rect(w).$}
\]

\subsection{Jeu de taquin as rewriting}
\label{SS:JeutaquinAsRewriting}

\subsubsection{Rules of the jeu de taquin}
\label{SSS:RulesJeuDetaquin}

The jeu de taquin map~$\rect$ is described by the union of rewriting systems $\Fr\Sr_n =\Lr\Sr_n\cup\Ir\Sr_n\cup\Tr\Sr_n$ whose sets of rules are defined as follows.
\begin{enumerate}[\bf i)]
\item~$\Lr\Sr_n$ the set of \emph{left-sliding} rules that move sub-columns to the left in the following two situations:
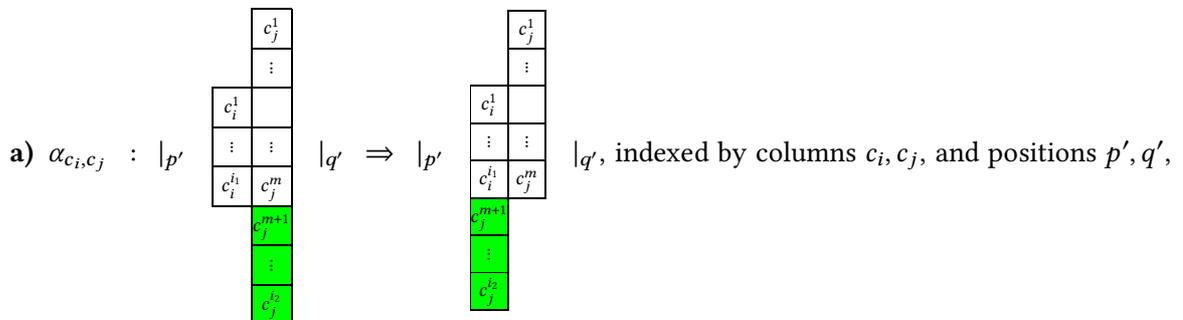
\begin{enumerate}[{\bf a)}]
\item $\alpha_{c_i,c_j}\;:\;$
$|_{p'}$
$\raisebox{1.5cm}{
\scalebox{0.7}{
\ytableausetup{mathmode, boxsize=1.9em}
\begin{ytableau}
\none & c^1_j\\
\none & \points\\
c^1_i & \empty \\
\points &\points\\
c^{i_1}_i& c^m_j\\
\none & *(GreenL) c^{m+1}_j \\
\none &*(GreenL)\points \\
\none& *(GreenL)c^{i_2}_j  
\end{ytableau}
}}$
$\;|_{q'}$
$\;\dfl\;$
$|_{p'}$
$\raisebox{1.5cm}{
\scalebox{0.7}{
\ytableausetup{mathmode, boxsize=1.8em}
\begin{ytableau}
\none & c^1_j\\
\none & \points\\
c^1_i & \empty\\
\points &\points\\
c^{i_1}_i& c^m_j\\
*(GreenL)c^{m+1}_j &\none \\
*(GreenL)\points & \none\\
*(GreenL)c^{i_2}_j & \none
\end{ytableau}
}}$
$\;|_{q'}$, 
indexed by columns $c_i,c_j$, and positions $p',q'$, such that~$(c_i^{i_1},c_{j}^{m})$ is a row,~$i_1\leq m<i_2$ and~$c_i|_{(i_1+i_2-m)}c_j\in\Scolc{n}$.
\item $\delta_{c_i,c_j}^{\alpha}\;:$
$|_{p'}$
$\raisebox{1.5cm}{
\scalebox{0.7}{
\ytableausetup{mathmode, boxsize=1.8em}
\begin{ytableau}
  c^1_i&\none\\
\points&\none\\
\empty&c_j^1\\
 \points&\points\\
 *(Yellow)c^m_i&\empty\\
 \points&\points\\
 \empty&*(YellowD)c^n_j  \\
\points &\points\\
*(BlueD) c^{i_1}_i& c^{q}_j \\
\none&\points  \\
\none& *(Blue)c^p_j \\
 \none&\points\\
\none&c^{i_2}_j 
\end{ytableau}
}}$
$\;|_{q'}$
$\;\dfl\;$
$|_{p'}$
$\raisebox{1.5cm}{
\scalebox{0.7}{
\ytableausetup{mathmode, boxsize=1.8em}
\begin{ytableau}
\none&c_j^1 \\
\none&\points\\
 c^1_i& \empty \\
\points&\points  \\
*(Yellow) c^m_i&*(YellowD) c^n_j \\
\points&\points \\
\empty&c_i^q  \\
\points&\points\\
*(BlueD) c^{i_1}_i& *(Blue) c^p_j \\
*(GreenL)c_{j}^{p+1}&\none\\
*(GreenL)\points&\none  \\
*(GreenL) c^{i_2}_j&\none 
\end{ytableau}
}}$
$|_{q'+(p-q)}$, 
indexed by columns~$c_i,c_j$, and positions~$p',q'$, such that~$c_i|_{(i_{1}+i_{2}-q)} c_j\notin \Scolc{n}$, and~$m$ is maximal such that~$c_i|_{(i_{1}+i_{2}-p)} c_j\in \Scolc{n}$ and~$i_1\leq p$.
\end{enumerate}
\item $\Ir\Sr_n$ is the set of \emph{insert-sliding} rules performing insertion in the following two situations:
\begin{enumerate}[{\bf a)}]
\item $\beta_{c_i,c_j}\;:\;$
$|_{p'}$
$\raisebox{1.5cm}{
\scalebox{0.7}{
\ytableausetup{mathmode, boxsize=1.8em}
\begin{ytableau}
\none & c^1_j\\
\none & \points\\
c^1_i & \empty\\
\points&\points\\
c^{l-1}_i & *(GreenL) c_j^m\\
c^{l}_i &c_j^{m+1}\\
\points&\points\\
c^{k}_i& c^{i_2}_j\\
\points & \none\\
c^{i_1}_i & \none
\end{ytableau}
}}$
$\;|_{q'}$
$\;\dfl\;$
$|_{p'}$
$\raisebox{1.5cm}{
\scalebox{0.7}{
\ytableausetup{mathmode, boxsize=1.8em}
\begin{ytableau}
\none & c^1_j\\
\none & \points\\
c^1_i & \empty\\
\points&\points\\
c^{l-1}_i &c_j^{m-1}\\
*(GreenL)c_j^m &c_j^{m+1}\\
c^{l}_i &c_j^{m+2}\\
\points&\points\\
c^{k-1}_i& c^{i_2}_j\\
\points & \none\\
c^{i_1}_i & \none
\end{ytableau}
}}$
$\;|_{q'}$, 
indexed by columns~$c_i,c_j$, and positions~$p',q'$, such that~\mbox{$c_i|_{k}c_j\in\Scolc{n}$} with~$1\leq k\leq i_1$ and~$k<i_2$, and~$l$ is minimal such that $(c_i^l, c_j^{m+1})$ is a row  and~$c^{l-1}_i< c^m_j <c_i^l$.
\item 
$\delta_{c_i,c_j}^{\beta}\;:$
$|_{p'}$
$\raisebox{1.5cm}{
\scalebox{0.7}{
\ytableausetup{mathmode, boxsize=1.8em}
\begin{ytableau}
*(Yellow) c^1_i&\none\\
\points&\none\\
\empty&c_j^1\\
\points&\points\\
\empty&*(YellowD)c^r_j \\
\points&\points\\
*(BlueD)c_i^p& \\
\points &\points\\
 c^{i_1}_i& c_j^q\\
\none&\points\\
\none&*(Blue)c^{i_2}_j 
\end{ytableau}
}}$
$\;|_{q'}$
$\;\dfl\;$
$|_{p'}$
$\raisebox{1.5cm}{
\scalebox{0.7}{
\ytableausetup{mathmode, boxsize=1.8em}
\begin{ytableau}
\none&c_j^1\\
\none&\points\\
*(Yellow)  c^1_i&c_j^{r-1}\\
c^2_i&*(YellowD)c_j^r \\
\points&\points  \\
c_i^{n-1}& c_j^{t-1} \\
*(GreenL) c_j^{t}&c_j^{t+1} \\
\points&\points\\
\empty&c_j^q\\
\points&\points\\
*(BlueD) c^{p}_i & *(Blue) c^{i_2}_j\\
\points&\none\\
c_i^{i_1}
\end{ytableau}
}}$
$|_{q'+(i_1+i_2-p-q)}$,
indexed by columns~$c_i,c_j$, and positions~$p',q'$, such that~$c_i|_{(i_{1}+i_{2}-q)} c_j\notin \Scolc{n}$, where~$p$ is maximal such that~$c_i|_{p} c_j\in \Scolc{n}$, and~$n$ is minimal such that~$(c_i^n, c_j^{t+1})$ is a row with~$c_i^{n-1}<c_j^t<c_i^n$.
\end{enumerate}
\item~$\Tr\Sr_n$ is the set of  \emph{left top sliding} rules that move columns to the top as follows:
\begin{enumerate}[{\bf a)}]
\item $\gamma_{c_i,c_j}\;:\;$
$|_{p'}$
\raisebox{1.5cm}{
\scalebox{0.7}{
\ytableausetup{mathmode, boxsize=1.8em}
\begin{ytableau}
\none &  c^1_j\\
\none & \points\\
\none & *(GreenD)c^m_j\\
\none & \points\\
*(GreenL)c^1_i & \empty \\
\points &\points\\
 c^{r}_i& *(BlueD) c^{i_2}_j\\
\points & \none\\
*(Blue)c^s_i & \none\\
\points& \none\\
c^{i_1}_i & \none
\end{ytableau}
}}
\;$|_{q'}$
$\;\dfl\;$
$|_{p'-(s-r)}$
\raisebox{1.5cm}{
\scalebox{0.7}{
\ytableausetup{mathmode, boxsize=1.8em}
\begin{ytableau}
\none & c^1_j\\
\none & \points\\
*(GreenL) c^1_i &*(GreenD) c^m_j\\
\points &\points\\
c_i^r & \empty\\
\points&\points\\
*(Blue) c^s_i & *(BlueD) c^{i_2}_j\\
\points & \none\\
c^{i_1}_i & \none
\end{ytableau}}}
\;$|_{q'}$, \quad
indexed by columns~$c_{i},c_j$ and positions~$p',q'$, such that~$c_{i}|_{r}c_j\in \Scolc{n}$ with~$1\leq r<i_2$, or~$c_i|_{r}c_{j}$ is not row connected with~$c_i^{1}\leq c_j^{i_2}$, and~$s$ is maximal  such that~$c_i|_sc_j\in\Scolc{n}$ and~$s\leq i_2$.
\item 
$\delta_{c_i,c_j} :  |_{p'}$
$\raisebox{1.5cm}{
\scalebox{0.7}{
\ytableausetup{mathmode, boxsize=1.8em}
\begin{ytableau}
*(Yellow)  c^1_i&\none\\
\points&\none\\
c_i^{q}&*(YellowD)c^1_j  \\
\points &\points\\
*(BlueD) c^{i_1}_i& c_j^p \\
\none&\points  \\
 \none&*(Blue)c_j^{i_1}\\
 \none&\points\\
\none&c^{i_2}_j\\
\end{ytableau}
}}$
$\;|_{q'}$
$\;\dfl\;$
$|_{p'}$
$\raisebox{1.5cm}{
\scalebox{0.7}{
\ytableausetup{mathmode, boxsize=1.8em}
\begin{ytableau}
 *(Yellow)  c^1_i&*(YellowD)c^1_j \\
\points&\points\\
 \empty&c_j^p \\
\points &\points\\
 c_i^q&\empty \\
\points &\points\\
*(BlueD) c^{i_1}_i& *(Blue)c_j^{i_1} \\
\none&\points  \\
 \none&c^{i_2}_j
\end{ytableau}
}}$
$\;|_{q'+(i_1-p)}$, 
indexed by columns~$c_i,c_j$, and positions~$p',q'$, such that~$c_i|_{(i_{1}+i_{2}-p)} c_j\notin \Scolc{n}$ and $c_i|_{i_{2}} c_j\in \Scolc{n}$, or $c_i|_{(i_{1}+i_{2}-p)} c_j\in \Scolc{n}$ and~$i_1>p$, or~$c_i|_{k}c_{j}$ is not row connected with~$k>i_1$ and~$c_i^{i_1}\leq c_j^1$.
\end{enumerate}
\end{enumerate}
In the sequel, if there is no possible confusion, we will omit the subscripts~$c_i,c_j$ in the notation of the rules. Moreover,  for any rule $\mu$ in $\Fr\Sr_n$, we will denote by~$\mu^\ast$ any composition of rewriting sequences involving the rules~$\mu$  and ending on a normal form with respect to~$\mu$.

\subsubsection{Example} 
The rectification of the skew tableau~$w$ from Example~\ref{Ex:Jeudetaquin} is computed with the following reduction of~$\Fr\Sr_n$: 
\[
w\:=\: \raisebox{0.45cm}{
\ytableausetup{smalltableaux}
\begin{ytableau}
\none & \none & *(GreenD)  1 & 2 \\
\none &*(GreenL) 1 & *(BlueD)  3 \\
1 & *(Blue) 2 \\
3
\end{ytableau}
\raisebox{-0.5cm}{$\;\odfl{\gamma_{c_2,c_3}}\;$}
\begin{ytableau}
\none &  1 & 1 & 2 \\
\none &*(GreenD) 2  &3 \\
*(GreenL)1 \\
3
\end{ytableau}
\raisebox{-0.5cm}{$\;\odfl{\gamma_{c_1,c_2}}\;$}
\begin{ytableau}
\none & 1& 1 & 2 \\
 1 &*(GreenL) 2& 3 \\
  3 
\end{ytableau}
\raisebox{-0.5cm}{$\;\odfl{\beta_{c_1,c_2}}\;$}
\begin{ytableau}
1 & 1& 1 & 2 \\
 2 & \none & *(GreenL)3 \\
3 
\end{ytableau}
\raisebox{-0.5cm}{$\;\odfl{\alpha_{c_2,c_3}}\;$}
\begin{ytableau}
1 & 1& 1 & 2 \\
2 & 3 \\
3 
\end{ytableau}
\raisebox{-0.35cm}{$\:=\:\rect(w)$.}
}
\]

\begin{theorem}
\label{T:RewritingPropertiesYoung}
The rewriting system $\Fr\Sr_n$ satisfies the following conditions: 
\begin{enumerate}[\bf i)]
\item $\Fr\Sr_n$ is convergent.
\item The normal form of any skew tableau with respect to~$\Fr\Sr_n$ is a Young tableau.
\end{enumerate}
For every word $w$ in $[n]^\ast$, we have
\begin{enumerate}[\bf i)]
\setcounter{enumi}{2}
\item $(\emptyset \insr{S_r} w) =   \rho_{\Fr\Sr_n}^\top([w]_s)$ and
$(w \insl{S_l} \emptyset)  =  \rho_{\Fr\Sr_n}^\perp([w]_s)$,
\item $C_{\Yrow{n}}(w)= \nf([w]_s,\Fr\Sr_n)$  and~$C_{\Ycol{n}}(w)= \nf([w]_s,\Fr\Sr_n)$.
\end{enumerate}
\end{theorem}

The rest of this subsection is devoted to the proof of this result. Lemmata~\ref{Lemma:TerminationJeuDeTaquin} and~\ref{Lemma:ConfluenceJeuDeTaquin} show that the rewriting system~$\Fr\Sr_n$ is convergent. As a consequence, we obtain that  the normal forms are Young tableaux. 
We prove in~\ref{SSS:InsertionsConstructorsNormalizations} that right and left Schensted's insertion algorithm coincide respectively with the leftmost and rightmost normalization strategy of~$\Fr\Sr_n$. Condition~{\bf iii)} and convergence of~$\Fr\Sr_n$ yield Condition~{\bf iv)}.

\begin{lemma}
\label{Lemma:SchenstedInerstionsAndNormalizations}
For any rule~$|c_{i_1}|c_{i_2}|\dfl |c_{j_1}|c_{j_2}|$ in~$\Fr\Sr_n$, the following equality~$|c_{i_1}\star_{S_r} c_{i_2}| = \rho_{\Fr\Sr_n}^\top(|c_{j_1}|c_{j_2}|)$ holds. Moreover, for all~$c_i=(c_i^1,\ldots, c_i^{i_1})$ and~$c_j=(c_j^1,\ldots, c_j^{i_2})$ in~$\Colo{n}$, we have
\begin{eqn}{equation}
\label{E:SchenstedInsertionLeftNormalization}
c_i\star_{S_r} c_j = \rho_{\Fr\Sr_n}^\top(c_i|_{1}c_j),
\qquad c_j\star_{S_l} c_i =  \rho_{\Fr\Sr_n}^\perp(c_i|_{1}c_j).
\end{eqn}
\end{lemma}
\begin{proof}
Prove that for any rule~$|c_{i_1}|c_{i_2}|\dfl |c_{j_1}|c_{j_2}|$ in~$\Fr\Sr_n$, we have~$|c_{i_1}\star_{S_r} c_{i_2}| = \rho_{\Fr\Sr_n}^\top(|c_{j_1}|c_{j_2}|)$. The rules~$\alpha_{c_i,c_j}$,~$\delta^{\alpha}_{c_i,c_j}$,~$\beta_{c_i,c_j}$ and~$\delta^{\beta}_{c_i,c_j}$ followed by~$\beta^\ast$ yield to the Young tableau~$|c_i\star_{S_r} c_j|$. Consider now the rule~$\gamma_{c_i,c_j}$. If~$|c_j|\leq |c_i|$ and~$c_i|_{|c_j|}c_j\in \Scol{n}$ then the target of~$\gamma_{c_i,c_j}$ is equal to~$|c_i\star_{S_r} c_j|$. If~$|c_i|< |c_j|$ and~$c_i|_{|c_j|}c_j\in \Scol{n}$, then the rule $\gamma_{c_i,c_j}$ is followed by the rule~$\alpha_{c_i,c_j}$ in order to obtain~$|c_i\star_{S_r} c_j|$. If~$c_i|_{|c_j|}c_j\notin \Scol{n}$ then the rule $\gamma_{c_i,c_j}$ followed by~$\alpha_{c_i,c_j}$ and then by~$\beta^\ast$, or only followed by~$\beta^\ast$ yield to~$|c_i\star_{S_r} c_j|$.
Consider finally the rule~$\delta_{c_i,c_j}$. If $|c_j|\leq |c_i|$ then the target~$\delta_{c_i,c_j}$ is equal to~$|c_i\star_{S_r} c_j|$. Otherwise, if $|c_i|< |c_j|$ then $\delta_{c_i,c_j}$ is followed by~$\alpha_{c_i,c_j}$ in order to obtain~$|c_i\star_{S_r} c_j|$.

Prove  the first equality of~(\ref{E:SchenstedInsertionLeftNormalization}) by induction on~$|c_j|$, the proof being similar for~$\star_{S_l}$. Suppose that~$|c_j|=1$, we consider the following two cases. If~$c_i^1\leq c_j^{1}$, then~$c_i|_{1}c_j$ is equal to~$c_i\star_{S_r} c_j$. If~$c_j^1< c_i^1$, then by applying~$\delta^\beta$ on~$c_i|_{1}c_j$ we obtain~$c_i\star_{S_r} c_j$.
Suppose the equality holds when~$|c_j| = i_2-1$, and prove it when $|c_j|=i_2$.
First consider the case when~$c_i^1\leq c_j^{i_2}$.
Suppose that by induction we have 
\[
\raisebox{-0.5cm}{$\rho_{\Fr\Sr_n}^\top\big($}
\scalebox{0.7}{
\ytableausetup{mathmode, boxsize=1.8em}
\begin{ytableau}
\none &c_j^{2}\\
\none &\points\\
c_i^1&c_j^{i_{2}}\\
\points\\
c_i^{i_1}
\end{ytableau}}
\raisebox{-0.5cm}{$\quad \big)\quad = \quad$}
\raisebox{0.1cm}{
\scalebox{0.7}{
\begin{ytableau}
x_1&y_1\\
\points&\points\\
\empty&y_t\\
\points\\
x_s
\end{ytableau}}
}
\raisebox{-0.5cm}{$\quad=c_i\star_{S_r} (c_j^{2},\ldots,c_j^{i_2})$,}
\]
where~$x_i$ and~$y_j$ are elements of~$c_i$ and~$c_j$ with~$t\leq s$. We prove that
$\raisebox{0cm}{$\rho_{\Fr\Sr_n}^\top\big(\;$}
\scalebox{0.7}{
\raisebox{1.5cm}{
\begin{ytableau}
\none&c_j^{1}\\
x_1&y_1\\
\points&\points\\
\empty &y_{t}\\
\points\\
x_s
\end{ytableau}
}}
\raisebox{0cm}{$\;\big)= c_i\star_{S_r} c_j$}$,
by considering the following two cases.

{\bf Case~1.}~$x_1\leq c_j^1$ and~$x_{k+1}\leq y_{k}$, for all~$k=1,\ldots, t-1$:
\[
\scalebox{0.7}{
\ytableausetup{mathmode, boxsize=1.8em}
\begin{ytableau}
\none&c_j^{1}\\
x_1&y_1\\
\points&\points\\
x_s &y_{t}\\
\end{ytableau}}
\raisebox{-0.6cm}{$\quad\odfl{\gamma}\quad$}
\scalebox{0.7}{
\begin{ytableau}
x_1&c_j^{1}\\
x_2&y_1\\
\points&\points\\
x_s &y_{t-1}\\
\none&y_t
\end{ytableau}}
\raisebox{-0.6cm}{$\quad\odfl{\alpha}\quad$}
\scalebox{0.7}{
\begin{ytableau}
x_1&c_j^{1}\\
x_2&y_1\\
\points&\points\\
x_s &y_{t-1}\\
y_t
\end{ytableau}}
\raisebox{-0.5cm}{$\quad= c_i\star_{S_r} c_j,\; \text{ for } \; s=t$}
\]
\[
\raisebox{-0.5cm}{$\text{ or }\quad\quad$}
\scalebox{0.7}{
\ytableausetup{mathmode, boxsize=1.8em}
\begin{ytableau}
\none&c_j^{1}\\
x_1&y_1\\
\points&\points\\
\empty &y_{t}\\
\points&\none\\
x_s
\end{ytableau}}
\raisebox{-0.6cm}{$\quad\odfl{\gamma}\quad$}
\scalebox{0.7}{
\begin{ytableau}
x_1&c_j^{1}\\
x_2&y_1\\
\points&\points\\
\empty &y_{t}\\
\points\\
x_s
\end{ytableau}
}
\raisebox{-0.5cm}{$\quad= c_i\star_{S_r} c_j,\; \text{ for } \; s>t.$}
\]
\noindent {\bf Case~2.}~$x_1\leq c_j^1$ (resp.~$c_j^1<x_1$)  and let~$x_i$ be minimal such that~$x_{i-1}<y_{i-1}<x_{i}$:
\[
\scalebox{0.7}{
\ytableausetup{mathmode, boxsize=1.8em}
\begin{ytableau}
\none&c_j^{1}\\
x_1&y_1\\
\points&\points\\
x_{i-1}&y_{i-1}\\
x_{i}&y_{i}\\
\points&\points\\
\empty &y_{t}\\
\points\\
x_s
\end{ytableau}}
\raisebox{-1.5cm}{$\quad\odfl{\beta}\quad$}
\scalebox{0.7}{
\begin{ytableau}
x_1&c_j^{1}\\
x_2&y_1\\
\points&\points\\
y_{i-1}&y_{i}\\
x_{i}&y_{i+1}\\
\points&\points\\
\empty &y_{t}\\
\points\\
x_s 
\end{ytableau}
}
\raisebox{-1cm}{$\quad= c_i\star_{S_r} c_j.$}
\qquad\qquad
\raisebox{-1cm}{$\big(\;\text{ resp. }\quad$}
\scalebox{0.7}{
\ytableausetup{mathmode, boxsize=1.8em}
\begin{ytableau}
\none&c_j^{1}\\
x_1&y_1\\
\points&\points\\
\empty &y_{t}\\
\points\\
x_s
\end{ytableau}}
\raisebox{-1cm}{$\quad\odfl{\beta}\quad$}
\scalebox{0.7}{
\begin{ytableau}
c_j^{1}&y_1\\
x_1&y_2\\
\points&\points\\
\empty&y_{t}\\
\points\\
x_s 
\end{ytableau}}
\raisebox{-1cm}{$\quad= c_i\star_{S_r} c_j.$}
\raisebox{-1cm}{$\;\big)$}
\]
Suppose finally that~$c_i^1>c_j^{i_2}$. We obtain:$
\scalebox{0.7}{
\ytableausetup{mathmode, boxsize=1.8em}
\raisebox{1cm}{
\begin{ytableau}
\none &c_j^{1}\\
\none &\points\\
c_i^1&c_j^{i_{2}}\\
\points\\
c_i^{i_1}
\end{ytableau}}}
\raisebox{0cm}{\quad$\odfl{\delta^\beta}$\quad}
\scalebox{0.7}{
\raisebox{1.2cm}{
\begin{ytableau}
\none &c_j^{1}\\
\none &\points\\
\none&c_j^{i_{2}-1}\\
c_j^{i_{2}}\\
c_i^1\\
\points\\
c_i^{i_1}
\end{ytableau}}}
\raisebox{0cm}{\quad$\odfl{\beta^\ast}$\quad}
\scalebox{0.7}{
\raisebox{1cm}{
\begin{ytableau}
c_j^{1}\\
\points\\
c_j^{i_{2}}\\
c_i^1\\
\points\\
c_i^{i_1}
\end{ytableau}}}
\raisebox{0cm}{$\quad= c_i\star_{S_r} c_j.\quad$}$
\end{proof}

\subsubsection{Proof of Theorem~\ref{T:RewritingPropertiesYoung} iii)}
\label{SSS:InsertionsConstructorsNormalizations}
We prove the first equality 
by induction on the number of columns in~$[w]_s$, the proof being similar for the insertion~$S_l$. When $[w]_s$ is of length $2$, then the equality is a consequence of Lemma~\ref{Lemma:SchenstedInerstionsAndNormalizations}. For $k\geq 3$, suppose that the equality holds for words of length~$k-1$, and consider~$[w]_s = c_1|_{1}\ldots|_{1}c_{k}$. By the induction hypothesis, we have~$(\emptyset \insr{S_r}w)   =   \rho_{\Fr\Sr_n}^\top(c_1|_{1}\ldots|_{1}c_{k-1})\star_{S_r}c_k$.
Let us show that 
\begin{eqn}{equation}
\label{E:LeftMostNormalizationJeuDetaquin}
\rho_{\Fr\Sr_n}^\top(c_1|_{1}\ldots|_{1}c_{k-1})\star_{S_r}c_k = \rho_{\Fr\Sr_n}^\top(c_1|_{1}\ldots|_{1}c_{k}).
\end{eqn} 
Since inserting~$c_k$ into~$\rho_{\Fr\Sr_n}^\top(c_1|_{1}\ldots|_{1}c_{k-1})$ consists into inserting its elements one by one from bottom to top, it suffices to prove~(\ref{E:LeftMostNormalizationJeuDetaquin}) for~$c_k=(x)$.
If~$x$ is bigger or equal than the last element~$x_1^{i_1}$ of the first row of~$\rho_{\Fr\Sr_n}^\top(c_1|_{1}\ldots|_{1}c_{k-1})$, then~$\rho_{\Fr\Sr_n}^\top(c_1|_{1}\ldots|_{1}c_{k-1})|_{1} \ytableausetup{smalltableaux}
\begin{ytableau}
x
\end{ytableau}$ is a Young tableau which is equal to~$\rho_{\Fr\Sr_n}^\top(c_1|_{1}\ldots|_{1}c_{k-1})\star_{S_r} \ytableausetup{smalltableaux}
\begin{ytableau}
x
\end{ytableau}$.
Otherwise, if~$x<x_1^{i_1}$, we first apply a rule~$\delta^\beta$ in order to slide the box containing~$x$ to the top of the one containing~$x_1^{i_1}$. We then apply the following reduction rules as shown in the following reduction diagrams. Note that, the elements in the colored boxes represent the ones that are bumped when inserting~$x$ into the tableau~$\rho_{\Fr\Sr_n}^\top(c_1|_{1}\ldots|_{1}c_{k-1})$.
\[
\scalebox{0.5}{
\ytableausetup{mathmode, boxsize=2em}
\begin{ytableau}
\none&\none&\none&\none&\none&\none&\none&\none&\none&\none&\none&\none&\none&\none&\none&\none&x\\
\none[\ldots]&x_1^{i_k}&\none[\ldots]&x_1^{k_1}&\none[\ldots]&x_1^{k_2}&\none[\ldots]&x_1^{i_l}&x_1^{i_{l}+1}&\none[\ldots]&x_1^{k_3}&x_1^{k_{3}+1}&\none[\ldots]&*(GreenL)x_1^{k_4}&x_1^{k_{4}+1}&\none[\ldots]& x_1^{i_{1}}\\
\none[\ldots]&x_2^{i_k}&\none[\ldots]&x_2^{k_1}&\none[\ldots]&x_2^{k_2}&\none[\ldots]&x_2^{i_{l}}&x_2^{i_{l}+1}&\none[\ldots]&*(GreenD)x_2^{k_3}&x_2^{k_{3}+1}&\none[\ldots]&x_2^{k_4}&x_2^{k_{4}+1}&\none[\ldots]\\
\none[\ldots]&x_3^{i_k}&\none[\ldots]&x_3^{k_1}&\none[\ldots]&x_3^{k_2}&\none[\ldots]&*(Blue)x_3^{i_l} \\
\none[\vdots]&\none[\vdots]&\none[\vdots]&\none[\vdots]&\none[\vdots]&\none[\vdots]&\none[\vdots]&\none[\vdots]\\
\none[\ldots]&x_{l-1}^{i_k}&\none[\ldots]&x_{l-1}^{k_1}&\none[\ldots]&*(BlueD)x_{l-1}^{k_2}&\none[\ldots]&x_{l-1}^{i_l}\\\none[\ldots]&x_l^{i_k}&\none[\ldots]&*(YellowD)x_l^{k_1}&\none[\ldots]&x_l^{k_2}&\none[\ldots]&x_l^{i_l}\\
\none[\ldots]&x_{l+1}^{i_k}
\end{ytableau}
}
\raisebox{-2cm}{$\quad\odfl{\beta^\ast}\quad$}
\scalebox{0.5}{
\ytableausetup{mathmode, boxsize=2em}
\begin{ytableau}
\none&\none&\none&\none&\none&\none&\none&\none&\none&\none&x_1^{k_3}&x_1^{k_{3}+1}&\none[\ldots]&x&x_1^{k_{4}+1}&\none[\ldots]& x_1^{i_{1}}\\
\none[\ldots]&x_1^{i_k}&\none[\ldots]&x_1^{k_1}&\none[\ldots]&x_1^{k_2}&\none[\ldots]&x_1^{i_l}&x_1^{i_{l}+1}&\none[\ldots]&*(GreenL)x_1^{k_4}&x_2^{k_{3}+1}&\none[\ldots]&x_2^{k_4}&x_2^{k_{4}+1}&\none[\ldots]\\
\none[\ldots]&x_2^{i_k}&\none[\ldots]&x_2^{k_1}&\none[\ldots]&x_2^{k_2}&\none[\ldots]&x_2^{i_{l}}&x_2^{i_{l}+1}&\none[\ldots]&*(GreenD)x_2^{k_3}\\
\none[\ldots]&x_3^{i_k}&\none[\ldots]&x_3^{k_1}&\none[\ldots]&x_3^{k_2}&\none[\ldots]&*(Blue)x_3^{i_l} \\
\none[\vdots]&\none[\vdots]&\none[\vdots]&\none[\vdots]&\none[\vdots]&\none[\vdots]&\none[\vdots]&\none[\vdots]\\
\none[\ldots]&x_{l-1}^{i_k}&\none[\ldots]&x_{l-1}^{k_1}&\none[\ldots]&*(BlueD)x_{l-1}^{k_2}&\none[\ldots]&x_{l-1}^{i_l}\\
\none[\ldots]&x_l^{i_k}&\none[\ldots]&*(YellowD)x_l^{k_1}&\none[\ldots]&x_l^{k_2}&\none[\ldots]&x_l^{i_l}\\
\none[\ldots]&x_{l+1}^{i_k}
\end{ytableau}
}\]
\[
\raisebox{-1.5cm}{$\odfl{\gamma^\ast}\quad$}
\scalebox{0.5}{
\ytableausetup{mathmode, boxsize=2em}
\begin{ytableau}
\none[\ldots]&x_1^{i_k}&\none[\ldots]&x_1^{k_1}&\none[\ldots]&x_1^{k_2}&\none[\ldots]&x_1^{i_l}&x_1^{i_{l}+1}&\none[\ldots]&
x_1^{k_3}&x_1^{k_{3}+1}&\none[\ldots]&x&x_1^{k_{4}+1}&\none[\ldots]& x_1^{i_{1}}\\
\none[\ldots]&x_2^{i_k}&\none[\ldots]&x_2^{k_1}&\none[\ldots]&x_2^{k_2}&\none[\ldots]&x_2^{i_{l}}&x_2^{i_{l}+1}&\none[\ldots]&
*(GreenL)x_1^{k_4}&x_2^{k_{3}+1}&\none[\ldots]&x_2^{k_4}&x_2^{k_{4}+1}&\none[\ldots]\\
\none[\ldots]&x_3^{i_k}&\none[\ldots]&x_3^{k_1}&\none[\ldots]&x_3^{k_2}&\none[\ldots]&*(Blue)x_3^{i_l} &\none&\none&*(GreenD)x_2^{k_3}\\
\none[\vdots]&\none[\vdots]&\none[\vdots]&\none[\vdots]&\none[\vdots]&\none[\vdots]&\none[\vdots]&\none[\vdots]\\
\none[\ldots]&x_{l-1}^{i_k}&\none[\ldots]&x_{l-1}^{k_1}&\none[\ldots]&*(BlueD)x_{l-1}^{k_2}&\none[\ldots]&x_{l-1}^{i_l}\\
\none[\ldots]&x_l^{i_k}&\none[\ldots]&*(YellowD)x_l^{k_1}&\none[\ldots]&x_l^{k_2}&\none[\ldots]&x_l^{i_l}\\
\none[\ldots]&x_{l+1}^{i_k}
\end{ytableau}
}
\raisebox{-1.5cm}{$\quad\odfl{\alpha^\ast}\quad$}
\scalebox{0.5}{
\ytableausetup{mathmode, boxsize=2em}
\begin{ytableau}
\none[\ldots]&x_1^{i_k}&\none[\ldots]&x_1^{k_1}&\none[\ldots]&x_1^{k_2}&\none[\ldots]&x_1^{i_l}&x_1^{i_{l}+1}&\none[\ldots]&
x_1^{k_3}&x_1^{k_{3}+1}&\none[\ldots]&x&x_1^{k_{4}+1}&\none[\ldots]& x_1^{i_{1}}\\
\none[\ldots]&x_2^{i_k}&\none[\ldots]&x_2^{k_1}&\none[\ldots]&x_2^{k_2}&\none[\ldots]&x_2^{i_{l}}&x_2^{i_{l}+1}&\none[\ldots]&
*(GreenL)x_1^{k_4}&x_2^{k_{3}+1}&\none[\ldots]&x_2^{k_4}&x_2^{k_{4}+1}&\none[\ldots]\\
\none[\ldots]&x_3^{i_k}&\none[\ldots]&x_3^{k_1}&\none[\ldots]&x_3^{k_2}&\none[\ldots]&*(Blue)x_3^{i_l}&*(GreenD)x_2^{k_3}\\
\none[\vdots]&\none[\vdots]&\none[\vdots]&\none[\vdots]&\none[\vdots]&\none[\vdots]&\none[\vdots]&\none[\vdots]\\
\none[\ldots]&x_{l-1}^{i_k}&\none[\ldots]&x_{l-1}^{k_1}&\none[\ldots]&*(BlueD)x_{l-1}^{k_2}&\none[\ldots]&x_{l-1}^{i_l}\\
\none[\ldots]&x_l^{i_k}&\none[\ldots]&*(YellowD)x_l^{k_1}&\none[\ldots]&x_l^{k_2}&\none[\ldots]&x_l^{i_l}\\
\none[\ldots]&x_{l+1}^{i_k}
\end{ytableau}
}
\]
\[
\raisebox{-2.cm}{$\quad\odfl{\delta^\beta}\quad$}
\scalebox{0.5}{
\ytableausetup{mathmode, boxsize=2em}
\begin{ytableau}
\none&\none&\none&\none&\none&\none&\none&x_1^{i_l}&x_1^{i_{l}+1}&\none&\none&\none&\none&\none&\none&\none&\none&\none\\
\none[\ldots]&x_1^{i_k}&\none[\ldots]&x_1^{k_1}&\none[\ldots]&x_1^{k_2}&\none[\ldots]&x_2^{i_l}&x_2^{i_{l}+1}&\none[\ldots]&
x_1^{k_3}&x_1^{k_{3}+1}&\none[\ldots]&x&x_1^{k_{4}+1}&\none[\ldots]& x_1^{i_{1}}\\
\none[\ldots]&x_2^{i_k}&\none[\ldots]&x_2^{k_1}&\none[\ldots]&x_2^{k_2}&\none[\ldots]&*(GreenD)x_2^{k_3}&\none&\none[\ldots]&
*(GreenL)x_1^{k_4}&x_2^{k_{3}+1}&\none[\ldots]&x_2^{k_4}&x_2^{k_{4}+1}&\none[\ldots]\\
\none[\ldots]&x_3^{i_k}&\none[\ldots]&x_3^{k_1}&\none[\ldots]&x_3^{k_2}&\none[\ldots]&*(Blue)x_3^{i_l} \\
\none[\vdots]&\none[\vdots]&\none[\vdots]&\none[\vdots]&\none[\vdots]&\none[\vdots]&\none[\vdots]&\none[\vdots]\\
\none[\ldots]&x_{l-1}^{i_k}&\none[\ldots]&x_{l-1}^{k_1}&\none[\ldots]&*(BlueD)x_{l-1}^{k_2}&\none[\ldots]&x_{l-1}^{i_l}\\
\none[\ldots]&x_l^{i_k}&\none[\ldots]&*(YellowD)x_l^{k_1}&\none[\ldots]&x_l^{k_2}&\none[\ldots]&x_l^{i_l}\\
\none[\ldots]&x_{l+1}^{i_k}
\end{ytableau}
}
\raisebox{-2cm}{$\quad\odfl{\beta^\ast}\quad$}
\scalebox{0.5}{
\ytableausetup{mathmode, boxsize=2em}
\begin{ytableau}
\none&\none&\none&x_1^{k_1}&\none[\ldots]&x_1^{i_l}&x_1^{i_{l}+1}&\none&\none&\none&\none&\none&\none&\none&\none&\none\\
\none[\ldots]&x_1^{i_k}&\none[\ldots]&x_2^{k_1}&\none[\ldots]&x_2^{i_l}&x_2^{i_{l}+1}&\none[\ldots]&
x_1^{k_3}&x_1^{k_{3}+1}&\none[\ldots]&x&x_1^{k_{4}+1}&\none[\ldots]& x_1^{i_{1}}\\
\none[\ldots]&x_2^{i_k}&\none[\ldots]&x_3^{k_1}&\none[\ldots]&*(GreenD)x_2^{k_3}&\none&\none[\ldots]&
*(GreenL)x_1^{k_4}&x_2^{k_{3}+1}&\none[\ldots]&x_2^{k_4}&x_2^{k_{4}+1}&\none[\ldots]\\
\none[\ldots]&x_3^{i_k}&\none[\ldots]&x_4^{k_1}&\none[\ldots]&x_4^{i_l} \\
\none[\vdots]&\none[\vdots]&\none[\vdots]&\none[\vdots]&\none[\vdots]&\none[\vdots]\\
\none[\ldots]&x_{l-1}^{i_k}&\none[\ldots]&*(BlueD)x_{l-1}^{k_2}&\none[\ldots]&x_l^{i_l}\\
\none[\ldots]&x_l^{i_k}&\none[\ldots]&*(YellowD)x_l^{k_1}\\
\none[\ldots]&x_{l+1}^{i_k}
\end{ytableau}
}
\]
\[
\raisebox{-1.5cm}{$\quad\odfl{\gamma^\ast}\quad$}
\scalebox{0.5}{
\ytableausetup{mathmode, boxsize=2em}
\begin{ytableau}
\none[\ldots]&x_1^{i_k}&\none[\ldots]&x_1^{k_1}&\none[\ldots]&x_1^{i_l}&x_1^{i_{l}+1}&\none&\none&\none&\none&\none&\none&\none&\none&\none\\
\none[\ldots]&x_2^{i_k}&\none[\ldots]&x_2^{k_1}&\none[\ldots]&x_2^{i_l}&x_2^{i_{l}+1}&\none[\ldots]&
x_1^{k_3}&x_1^{k_{3}+1}&\none[\ldots]&x&x_1^{k_{4}+1}&\none[\ldots]& x_1^{i_{1}}\\
\none[\ldots]&x_3^{i_k}&\none[\ldots]&x_3^{k_1}&\none[\ldots]&*(GreenD)x_2^{k_3}&\none&\none[\ldots]&
*(GreenL)x_1^{k_4}&x_2^{k_{3}+1}&\none[\ldots]&x_2^{k_4}&x_2^{k_{4}+1}&\none[\ldots]\\
\none[\ldots]&x_4^{i_k}&\none[\ldots]&x_4^{k_1}&\none[\ldots]&x_4^{i_l} \\
\none[\vdots]&\none[\vdots]&\none[\vdots]&\none[\vdots]&\none[\vdots]&\none[\vdots]\\
\none[\ldots]&x_{l}^{i_k}&\none[\ldots]&*(BlueD)x_{l-1}^{k_2}&\none[\ldots]&x_l^{i_l}\\
\none[\ldots]&x_{l+1}^{i_k}&\none&*(YellowD)x_l^{k_1}
\end{ytableau}
}
\raisebox{-1.5cm}{$\quad\odfl{\alpha^\ast}\quad$}
\scalebox{0.5}{
\ytableausetup{mathmode, boxsize=2em}
\begin{ytableau}
\none[\ldots]&x_1^{i_k}&x_1^{i_{k}+1}&\none[\ldots]&x_1^{k_1}&\none[\ldots]&x_1^{i_l}&x_1^{i_{l}+1}&\none&\none&\none&\none&\none&\none&\none&\none&\none\\
\none[\ldots]&x_2^{i_k}&x_2^{i_{k}+1}&\none[\ldots]&x_2^{k_1}&\none[\ldots]&x_2^{i_l}&x_2^{i_{l}+1}&\none[\ldots]&
x_1^{k_3}&x_1^{k_{3}+1}&\none[\ldots]&x&x_1^{k_{4}+1}&\none[\ldots]& x_1^{i_{1}}\\
\none[\ldots]&x_3^{i_k}&x_3^{i_{k}+1}&\none[\ldots]&x_3^{k_1}&\none[\ldots]&*(GreenD)x_2^{k_3}&\none&\none[\ldots]&
*(GreenL)x_1^{k_4}&x_2^{k_{3}+1}&\none[\ldots]&x_2^{k_4}&x_2^{k_{4}+1}&\none[\ldots]\\
\none[\ldots]&x_4^{i_k}&x_4^{i_{k}+1}&\none[\ldots]&x_4^{k_1}&\none[\ldots]&x_4^{i_l} \\
\none[\vdots]&\none[\vdots]&\none[\vdots]&\none[\vdots]&\none[\vdots]&\none[\vdots]\\
\none[\ldots]&x_{l}^{i_k}&x_{l}^{i_{k}+1}&\none[\ldots]&*(BlueD)x_{l-1}^{k_2}&\none[\ldots]&x_l^{i_l}\\
\none[\ldots]&x_{l+1}^{i_k}&*(YellowD)x_l^{k_1}
\end{ytableau}
}
\]
\[
\raisebox{-1.5cm}{$\quad\odfl{\delta^\ast}\quad$}
\scalebox{0.5}{
\ytableausetup{mathmode, boxsize=2em}
\begin{ytableau}
\none[\ldots]&x_1^{i_k}&x_1^{i_{k}+1}&\none[\ldots]&x_1^{k_1}&\none[\ldots]&x_1^{i_l}&x_1^{i_{l}+1}&\none[\ldots]&
x_1^{k_3}&x_1^{k_{3}+1}&\none[\ldots]&x&x_1^{k_{4}+1}&\none[\ldots]& x_1^{i_{1}}\\
\none[\ldots]&x_2^{i_k}&x_2^{i_{k}+1}&\none[\ldots]&x_2^{k_1}&\none[\ldots]&x_2^{i_l}&x_2^{i_{l}+1}&\none[\ldots]&
*(GreenL)x_1^{k_4}&x_2^{k_{3}+1}&\none[\ldots]&x_2^{k_4}&x_2^{k_{4}+1}&\none[\ldots]\\
\none[\ldots]&x_3^{i_k}&x_3^{i_{k}+1}&\none[\ldots]&x_3^{k_1}&\none[\ldots]&*(GreenD)x_2^{k_3}&\none\\
\none[\ldots]&x_4^{i_k}&x_4^{i_{k}+1}&\none[\ldots]&x_4^{k_1}&\none[\ldots]&x_4^{i_l} \\
\none[\vdots]&\none[\vdots]&\none[\vdots]&\none[\vdots]&\none[\vdots]&\none[\vdots]\\
\none[\ldots]&x_{l}^{i_k}&x_{l}^{i_{k}+1}&\none[\ldots]&*(BlueD)x_{l-1}^{k_2}&\none[\ldots]&x_l^{i_l}\\
\none[\ldots]&x_{l+1}^{i_k}&*(YellowD)x_l^{k_1}
\end{ytableau}
}
\]
The resulted Young tableau is equal to~$\rho_{\Fr\Sr_n}^\top(c_1|_{1}\ldots|_{1}c_{k-1})\star_{S_r} \ytableausetup{smalltableaux}
\begin{ytableau}
x
\end{ytableau}$, showing the claim.

\begin{lemma}
The rewriting system $\Fr\Sr_n$ is terminating.
\label{Lemma:TerminationJeuDeTaquin}
\end{lemma}

\begin{proof}
We prove that for any reduction~$w \dfl w'$ with respect to~$\Fr\Sr_n$, we have $w\ll_{tl} w'$ for the order~$\ll_{tl}$ defined~\ref{SSS:TopLEftSlidingOrder}. 
If~$w \dfl w'$ is a reduction with respect to~$\Lr\Sr_n$, then~$||w|| = ||w'||$ and~$tl(w)\prec_{revlex} tl(w')$, showing that~$w\ll_{tl} w'$. Suppose now that the reduction is with respect to~$\Ir\Sr_n$. There are two cases depending on the number of columns in the targets of the rules~$\beta$ and~$\delta^\beta$. If the targets consist only of one column
then~$||w||<||w'||$. If they consist of two columns then~$||w||= ||w'||$ and~\mbox{$tl(w)\prec_{revlex} tl(w')$.} Then, if $w\dfl w'$ is a reduction with respect to~$\Ir\Sr_n$, we obtain~$w\ll_{tl} w'$. Finally, for any reduction~\mbox{$w\dfl w'$} with respect to~$\Tr\Sr_n$, we have~$||w||=||w'||$,~$tl(w) = tl(w')$ and  $d^\top(w) \prec_{lex} d^\top(w')$, showing that~$w\ll_{tl} w'$.  
\end{proof}

\begin{lemma}
The rewriting system $\Fr\Sr_n$ is confluent.
\label{Lemma:ConfluenceJeuDeTaquin}
\end{lemma}

\begin{proof}
Following Lemma~\ref{L:Convergence}, we prove that the rewriting system~$\Fr\Sr_n$ is confluent by showing the confluence of all its critical branchings.
Consider first the  rewriting system~$\Srs(\Colo{n}, \Yrow{n})$ whose rules  are of the form~$\gamma_{c,c'} : |c|c'| \dfl |c \star_{S_{r}} c'|$, for all~$c, c'$ in~$\Colo{n}$ such that~$c|_\glueY c' \neq c \star_{S_{r}} c'$.
Prove that starting from a string of columns consisting of three columns~$|c_i|c_j|c_k|$, we lead to the Young tableau~$|c_i\star_{S_r} c_j\star_{S_r}c_k|$ after applying at most three steps of reductions with respect~$\Srs(\Colo{n}, \Yrow{n})$ starting from the left or from the right. We prove this result using Sch\"{u}tzenberger's involution on columns as in~{\cite[Remark~3.2.7]{HageMalbos17}}. 
In one hand, by definition of Schensted's insertion~$S_r$, starting from~$|c_i|c_j|c_k|$, we lead to~$|c_i\star_{S_r} c_j\star_{S_r}c_k|$ after applying at most three steps of reductions with respect~$\Srs(\Colo{n}, \Yrow{n})$ starting from the left. That is, we have
\[
|c_i| c_j| c_k|
  \odfl{\gamma_{c_i,c_j}| c_k}
|c_n|_\glueY c_{n'}| c_k|
  \odfl{c_n|\gamma_{c_{n'},c_k}}
 | c_n| c_s|_\glueY c_{s'}|\odfl{\gamma_{c_{n},c_{s}}| c_{s'}}
|c_{i}\star_{S_r}c_{j}\star_{S_r}c_{k}|.\]
In an other hand, we have
\[
|c_i | c_j| c_k|\odfl{c_i| \gamma_{c_j,c_k}} |c_i| (c_j\star_{S_r}c_k)| = |c_i| c_l|_\glueY c_{l'}|\odfl{\gamma_{c_i,c_l}| c_{l'}} |(c_i \star_{S_r}c_l)| c_{l'}|=|c_m|_\glueY c_{m'}| c_{l'}|\odfl{c_m| \gamma_{c_{m'},c_{l'}}} |c_m| (c_{m'}\star_{s_r}c_{l'})|.
\]
Let us show that~$|c_m| (c_{m'}\star_{s_r}c_{l'})|=|c_{i}\star_{S_r}c_{j}\star_{S_r}c_{k}|$. By applying the involution on tableaux, we obtain
\[
|c_k^{\ast}| c_j^{\ast}| c_i^{\ast}|\dfll  
|(c_k^{\ast}\star_{S_r}c_j^{\ast})| c_i^{\ast}|
=|c_{l'}^{\ast}|_\glueY c_l^{\ast}| c_i^{\ast}|
\dfll |c_{l'}^{\ast}| (c_l^{\ast}\star_{S_r}c_i^{\ast})|
= |c_{l'}^{\ast}| c_{m'}^{\ast}|_\glueY c_m^{\ast}|
\dfll |(c_{l'}^{\ast}\star_{S_r}c_{m'}^{\ast})| c_m^{\ast}|.
\]
By definition of $S_r$, we have~$|c_k^{\ast}\star_{S_r}c_j^{\ast}\star_{s_r}c_i^{\ast}| = |(c_{l'}^{\ast}\star_{S_r}c_{m'}^{\ast})| c_m^{\ast}|$. Since~$(c_k^{\ast}\star_{s_r}c_j^{\ast}\star_{S_r}c_i^{\ast}) = (c_i\star_{S_r}c_j\star_{S_r}c_k)^{\ast}$, we deduce that~$|(c_i\star_{S_r}c_j\star_{S_r}c_k)^{\ast}| = |(c_{l'}^{\ast}\star_{S_r}c_{m'}^{\ast})| c_m^{\ast}|$. Finally, by applying the involution on tableaux, we obtain~$|(c_i\star_{S_r}c_j\star_{S_r}c_k)|=|c_m|(c_{m'}\star_{S_r}c_{l'})|$.

Following Lemma~\ref{Lemma:SchenstedInerstionsAndNormalizations}, for any rule~$|c_{i_1}|c_{i_2}|\dfl |c_{j_1}|c_{j_2}|$ in~$\Fr\Sr_n$, we have~$|c_{i_1}\star_{S_r} c_{i_2}| = \rho_{\Fr\Sr_n}^\top(|c_{j_1}|c_{j_2}|)$.  Hence, any critical branching of~$\Tr\Sr_n$ has the following confluence diagram
\[
\hspace{-1cm}
\xymatrix @C=2.5em @R=0.5em{
&
|c_{j}|c_{j'}|c_{i''}|
  \ar@2  [r] ^-{\rho_{\Fr\Sr_n}^\top (c_{j}|c_{j'})}
&
  |c_{k}|_{\glueY}  c_{k'}|c_{i''}|
  \ar@2  [r] ^-{\rho_{\Fr\Sr_n}^\top (c_{k'}|c_{i''})}
&
 |c_{k}|c_{l}|_{\glueY} c_{l'}|
\ar@2  [r] ^-{\rho_{\Fr\Sr_n}^\top (c_{k}|c_{l})}
&
 |c_{k'}|_{\glueY} c_{l''}|c_{l'}|
\ar@2 @/^1ex/ [dr] ^{\delta_{c_{l''},c_{l'}}}
\\
|c_{i}|c_{i'}|c_{i''}|
  \ar@2 @/^1ex/ [ur] ^-{\epsilon_1}
  \ar@2 @/_1ex/[dr] _-{\epsilon_2}
&&&&&{|c_{i}\star_{S_r}c_{i'}\star_{S_r}c_{i''}|}
\\
&
|c_{i}|c_{m}|c_{m'}|
  \ar@2 [r] _-{\rho_{\Fr\Sr_n}^\top (c_{m}|c_{m'})}
&
|c_{i}|c_{n}|_{\glueY} c_{n'}|
 \ar@2 [r] _-{\rho_{\Fr\Sr_n}^\top (c_{i}|c_{n})}
&
|c_{k'}|_{\glueY} c_{s'}|c_{n'}|
\ar@2 [r] _-{\rho_{\Fr\Sr_n}^\top (c_{s'}|c_{n'})}
&
|c_{k'}|c_{l''}|_{\glueY} c_{l'}|
\ar@2  @/_1ex/[ur] _-{\gamma_{c_{k'},c_{l''}}}
}
\]
where $\epsilon_1$ and $\epsilon_2$ are $\Tr\Sr_n$-reductions and where some indicated  rules can correspond to identities,  such that~$c_{k}|_{\glueY}  c_{k'} = c_{j}\star_{S_r}c_{j'}$,~$c_{l}|_{\glueY}  c_{l'} = c_{k'}\star_{S_r}c_{i''}$,~$c_{k'}|_{\glueY}  c_{l''} = c_{k}\star_{S_r}c_{l}$,~\mbox{$c_{n}|_{\glueY}  c_{n'} = c_{m}\star_{S_r}c_{m'}$,}~\mbox{$c_{k'}|_{\glueY}  c_{s'} = c_{i}\star_{S_r}c_{n}$} and~$c_{l''}|_{\glueY}  c_{l'} =c_{s'}\star_{S_r} c_{n'}$. 
\end{proof}

\subsection{Jeu de taquin as morphism of monoids}
\label{SS:JeuDeTaquinMorphism}

In this subsection, we prove the compatibility of the rewriting system~$\Fr\Sr_n$ with the plactic congruence. 

\begin{theorem}
\label{T:MorphismRect}
The rectification map $\rect:\dSkew{n}\to \Young{n}$ is a surjective map that satisfies the following two properties:
\begin{enumerate}[\bf i)]
\item for any rule~$d \dfl d'$ in~$\Fr\Sr_n$, we have~$R_{SW}(d) \approx_{\Pl_n} R_{SW}(d')$, 
\item for all $w,w'\in [n]^\ast$, $w \approx_{\Pl_n} w'$ implies $\nf([w]_s,\Fr\Sr_n) = \nf([w']_s,\Fr\Sr_n)$.
\end{enumerate}
\end{theorem}

\begin{proof}
The surjectivity of~$\rect$ is a consequence of the relation~$\rect([R_{SW}(d)]_s) =~d$, that holds for any $d$ in~$\Young{n}$.
Prove first that  for any rule~$d \dfl d'$ in~$\Fr\Sr_n$, we have~$R_{SW}(d) \approx_{\Pl_n} R_{SW}(d')$. It suffices to show that~$R_{SW}(s(\eta)) \approx_{\Pl_n} R_{SW}(t(\eta))$, for every rule~$\eta$ in $\Fr\Sr_n$.
This is obvious for~$\gamma$ and~$\delta$. For the rule~$\alpha$, consider~$R_{SW}(s(\alpha))=c_i^{i_1}\ldots c_i^{1}c_j^{i_2}\ldots c_j^{m+1}c_j^{m}\ldots c_j^1$ and~$R_{SW}(t(\alpha))=c_j^{i_2}\ldots c_j^{m+1}c_i^{i_1}\ldots c_i^{1}c_j^{m}\ldots c_j^1$. On one hand, we have

$\begin{array}{rl}
& c_i^{i_1}\ldots {\bf c_i^{2}c_i^{1}c_j^{i_2}}\ldots c_j^{m+1}c_j^{m}\ldots c_j^1 \overset{(\ref{E:KnuthRelations})}{=}\ldots \overset{(\ref{E:KnuthRelations})}{=} c_i^{i_1}c_j^{i_2}c_i^{i_{1}-1}\ldots {\bf c_i^{2}c_i^{1}c_j^{i_{2}-1}}\ldots c_j^{m+1}c_j^{m}\ldots c_j^1\\
&\newline\overset{(\ref{E:KnuthRelations})}{=}\ldots \overset{(\ref{E:KnuthRelations})}{=}  {\bf c_i^{i_1}c_j^{i_2}c_j^{i_{2}-1}}c_i^{i_{1}-1}\ldots  c_i^{2}c_i^{1}c_j^{i_{2}-2}\ldots c_j^{m+1}c_j^{m}\ldots c_j^1   
\overset{(\ref{E:KnuthRelations})}{=}\\
&c_j^{i_2}c_i^{i_1}c_j^{i_{2}-1}c_i^{i_{1}-1}\ldots  {\bf c_i^{2}c_i^{1}c_j^{i_{2}-2}}\ldots c_j^{m+1}c_j^{m}\ldots c_j^1\overset{(\ref{E:KnuthRelations})}{=}\ldots\overset{(\ref{E:KnuthRelations})}{=} c_j^{i_2}\ldots c_j^{m+1}c_j^{m}c_i^{i_1}\ldots c_i^{1}c_j^{m-1}\ldots c_j^1.
\end{array}$

\noindent In an other hand, we have~$c_j^{i_2}\ldots c_j^{m+1}c_i^{i_1}\ldots {\bf c_i^{2}c_i^{1}c_j^{m}}\ldots c_j^1\overset{(\ref{E:KnuthRelations})}{=}\ldots\overset{(\ref{E:KnuthRelations})}{=} c_j^{i_2}\ldots c_j^{m+1}c_j^{m}c_i^{i_1}\ldots c_i^{1}c_j^{m-1}\ldots c_j^1$.
Hence, $R_{SW}(s(\alpha))\approx_{\Pl_n} R_{SW}(t(\alpha))$. Similarly, we show the property for rules~$\beta$,~$\delta^\alpha$ and~$\delta^\beta$.

Prove now that for all $w,w'\in [n]^\ast$, $w \approx_{\Pl_n} w'$ implies $\nf([w]_s,\Fr\Sr_n) = \nf([w']_s,\Fr\Sr_n)$. Since~$\Fr\Sr_n$ is convergent, we show that for all~$w,w'\in~[n]^\ast$, $w\approx_{\Pl_n} w'$ implies $[w]_s\approx_{\Fr\Sr_n} [w']_s$. 
Suppose first that~$w =ux_1{\bf xzy}y_1v$ and~$w' =ux_1{\bf zxy}y_1v$,  for all~$1\leq x\leq y<z\leq n$,~$u,v\in~[n]^\ast$ and~$x_1, y_1\in~[n]$, and show that~$[w]_s\approx_{\Fr\Sr_n} [w']_s$. We consider the following  cases:

\noindent \emph{\bf Case 1.}~$x_1\leq x$ and~$y\leq y_1$.
\[
\ytableausetup{mathmode, boxsize=1.2em}
\xymatrix @C=0.4em @R=0.4em{
{ux_1 zxyy_1v}
   \ar@2[r] ^-{}
     \ar@1 [d]_-{[w']_s}
& {ux_1 xzyy_1v}
     \ar@1 [d] ^-{[w]_s}
\\
     \raisebox{-0.15cm}{$[ux_1]_s \;\big|\;$}
\scalebox{0.7}{
\begin{ytableau}
x&y\\
z
\end{ytableau}}
 \raisebox{-0.15cm}{$\;\big|\;[y_{1}v]_s$}
& 
\raisebox{-0.15cm}{$[ux_1]_s \;\big|\;$}
\scalebox{0.7}{\begin{ytableau}
\none&y\\ 
x&z
\end{ytableau}}
\raisebox{-0.15cm}{$\;\big|\;[y_{1}v]_s$}
\ar@2 [d] ^-{\gamma}
\\
&\raisebox{-0.15cm}{$[ux_1]_s \;\big|\;$}
\scalebox{0.7}{
\begin{ytableau}
x&y\\ 
\none&z
\end{ytableau}}
\raisebox{-0.15cm}{$\;\big|\;[y_{1}v]_s$}
 \ar@2 [ul] ^-{\alpha}
   }
   \]
\noindent \emph{\bf Case 2.}~$x_1\leq x$ and~$y> y_1$. Suppose that~$v=y_2\ldots y_{q}y'v'$ such that~$y_1>y_2>\ldots>y_{q}$ and~$y_q\leq y'$.
\[
\xymatrix @C=0.1em @R=-0.6em{
{ux_1 zxyy_1v}
   \ar@2 [r] ^-{}
     \ar@1 [d]_-{[w']_s}
& {ux_1 xzyy_1v}
     \ar@1 [d] ^-{[w]_s}
\\
     \raisebox{-0.25cm}{$[ux_1]_s \;\big|\;$}
\scalebox{0.7}{\begin{ytableau}
\none&\none&y_q\\
\none&\none&\points\\
\none&\none&y_1\\
\none&x&y\\
\none &z
\end{ytableau}}
 \raisebox{-0.25cm}{$\;\big|\;[y'v']_s$}
\ar@2 [d] ^-{\beta^\ast}
& 
\raisebox{-0.25cm}{$[ux_{1}]_s \;\big|\;$}
\scalebox{0.7}{\begin{ytableau}
\none&y_q\\
\none&\points\\
\none&y_1\\
\none&y\\ 
x&z
\end{ytableau}}
\raisebox{-0.25cm}{$\;\big|\;[y'v']_s$}
\ar@2 [d] ^-{\gamma}
\\
\raisebox{-1cm}{$[ux_{1}]_s \;\big|\;$}
\scalebox{0.7}{\begin{ytableau}
\none&y_q\\
\none&\points\\
x&y_i\\
y_{i+1}&\none\\
\points&\none\\
y_1&\none\\
y&\none\\ 
z&\none
\end{ytableau}}
\raisebox{-1cm}{$\;\big|\;[y'v']_s$}
&
\raisebox{-1cm}{$[ux_{1}]_s \;\big|\;$}
\scalebox{0.7}{\begin{ytableau}
\none&y_q\\
\none&\points\\
x&y_i\\
\none&\points\\
\none&y_1\\
\none&y\\ 
\none&z
\end{ytableau}}
\raisebox{-1cm}{$\;\big|\;[y'v']_s$}
\ar@2 [l] ^-{\alpha}
 }
\]
\noindent \emph{\bf Case 3.}~$x<z<x_1$ and~$y\leq y_1$.
Suppose that~$u=u'x'x_p\ldots x_1$ such that~$x'\leq x_p$ and~$x_p> \ldots >x_1$.
\[
\xymatrix @C=0.5em @R=0.3em{
{ux_1 zxyy_1v}
   \ar@2 [r] ^-{}
     \ar@1 [d]_-{[w']_s}
& {ux_1 xzyy_1v}
     \ar@1 [d] ^-{[w]_s}
\\
 \raisebox{-0.15cm}{$[u'x']_s \;\big|\;$}
\scalebox{0.7}{
\begin{ytableau}
x&y\\
z\\
x_1\\
\points\\
x_p
\end{ytableau}}
\raisebox{-0.15cm}{$\;\big|\;[y_{1}v]_s$}
& 
\raisebox{-0.15cm}{$[u'x']_s \;\big|\;$}
\scalebox{0.7}{
\begin{ytableau}
\none&y\\ 
x&z\\
x_1\\
\points\\
x_p
\end{ytableau}}
\raisebox{-0.15cm}{$\;\big|\;[y_{1}v]_s$}
\ar@2 [l] ^-{\beta}
}
\]
The case~$x<z<x_1$ and~$y> y_1$ is studied in the same way.

\noindent \emph{\bf Case 4.}~$x<x_{1}\leq z$ and~$y>y_1$. We study similarly the case~$x<x_{1}\leq z$ and~$y\leq y_1$.
Suppose that~$u=u'x'x_p\ldots x_1$ and~$v=y_2\ldots y_{q}y'v'$  such that~$x'\leq x_p>\ldots>x_1$ and~$y_1>\ldots>y_{q}\leq y'$. 
\[
\xymatrix @C=0.5em @R=0.3em{
{ux_1 zxyy_1v}
   \ar@2 [r] ^-{}
     \ar@1 [d]_-{[w']_s}
& {ux_1 xzyy_1v}
     \ar@1 [d] ^-{[w]_s}
\\
     \raisebox{-0.15cm}{$[u'x']_s \;\big|\;$}
\scalebox{0.7}{
\begin{ytableau}
\none&\none&y_q\\
\none&\none&\points\\
\none&\none&y_1\\
\none&x&y\\
 x_1&z\\
\points&\none\\
x_p&\none
\end{ytableau}}
 \raisebox{-0.15cm}{$\;\big|\;[y'v']_s$}
\ar@2 [d] ^-{\beta}
& 
\raisebox{-0.15cm}{$[u'x']_s \;\big|\;$}
\scalebox{0.7}{
\begin{ytableau}
\none&y_q\\
\none&\points\\
\none&y_1\\
\none&y\\ 
x&z\\
x_1&\none\\
\points&\none\\
x_p&\none
\end{ytableau}}
\raisebox{-0.15cm}{$\;\big|\;[y'v']_s$}
\\
\raisebox{-0.15cm}{$[u'x']_s \;\big|\;$}
\scalebox{0.7}{
\begin{ytableau}
\none&\none&y_q\\
\none&\none&\points\\
\none&\none&y_1\\
x&z&y\\
 x_1\\
\points&\none\\
x_p&\none
\end{ytableau}}
 \raisebox{-0.15cm}{$\;\big|\;[y'v']_s$}
\ar@2 [r] ^-{\delta^\beta}
&
\raisebox{-0.15cm}{$[u'x']_s \;\big|\;$}
\scalebox{0.7}{
\begin{ytableau}
\none&\none&y_q\\
\none&\none&\points\\
\none&\none&y_1\\
\none &y\\
x&z\\
 x_1\\
\points&\none\\
x_p&\none
\end{ytableau}}
 \raisebox{-0.15cm}{$\;\big|\;[y'v']_s$}
   \ar@2 [u] _-{\beta^\ast}
 }
\]
Suppose now that~$w =uy_{1}{\bf yzx}x_{1}v$ and~$w' =uy_{1}{\bf yxz}x_{1}v$, for all~$1\leq x<y\leq z\leq n$,~$u,v\in~[n]^\ast$ and~$x_1, y_1\in~[n]$, and show that~$[w]_s\approx_{\Fr\Sr_n} [w']_s$.  If~$y_{1}\leq y$ and~$x\leq x_1$, then 
\[
\xymatrix @C=0.5em @R=0.7em{
{uy_{1}yzxx_{1}v}
   \ar@2 [r] ^-{}
     \ar@1 [d]_-{[w]_s}
& {uy_{1}yxzx_{1}v}
     \ar@1 [d] ^-{[w']_s}
\\
 \raisebox{-0.15cm}{$[uy_{1}]_s\;\big|\;$}
\scalebox{0.7}{
\begin{ytableau}
\none&x\\
y&z
\end{ytableau}}
 \raisebox{-0.15cm}{$\;\big|\;[x_{1}v]_s$}
\ar@2 [r] _-{\beta}
& 
 \raisebox{-0.15cm}{$[uy_{1}]_s\;\big|\;$}
\scalebox{0.7}{
\begin{ytableau}
x&z\\
y
\end{ytableau}}
 \raisebox{-0.15cm}{$\;\big|\;[x_{1}v]_s$}
}
 \]
The cases~($y_{1}> y$ and~$x> x_1$),~($y_{1}\leq y$ and~$x> x_1$) and~($y_{1}> y$ and~$x\leq x_1$) are studied similarly.
\end{proof}

As a consequence of Theorem~\ref{T:RewritingPropertiesYoung} and Theorem~\ref{T:MorphismRect}, we recover that the set of Young tableaux satisfies the cross-section property for the plactic monoid, we also deduce the commutation of Schensted's left and right insertion algorithm and that  the rectification map defines a surjective morphism of monoids between the sets~$\dSKrow{n}$ and~$\Yrow{n}$ equipped with the insertion products~$\star_{I_r^a}$ and~$\star_{S_r}$.

\begin{corollary}[cross-section property]
\label{C:Consequense1}
The following conditions hold
\begin{enumerate}[{\bf i)}]
\item  for all~$w,w'$ in~$[n]^\ast$, we have~$w\approx_{\Pl_n} w'$ if and only if~$C_{\Yrow{n}}(w) =C_{\Yrow{n}}(w')$,
\item the equality $C_{\Yrow{n}}(R_{SW}(d)) = C_{\Yrow{n}}(R_{SW}(\rect(d)))$ holds in~$\Young{n}$,  for any $d$ in $\dSkew{n}$. 
\end{enumerate}
\end{corollary}

\begin{proof}
Prove Condition~{\bf i)}. 
Consider two words~$w$ and~$w'$ over~$[n]$. If $w\approx_{\Pl_n} w'$, then by Theorem~\ref{T:MorphismRect} we have~$\nf([w]_s,\Fr\Sr_n) = \nf([w']_s,\Fr\Sr_n)$ and thus~$C_{\Yrow{n}}(w) =C_{\Yrow{n}}(w')$  by Theorem~\ref{T:RewritingPropertiesYoung}. Suppose now that~$C_{\Yrow{n}}(w) =C_{\Yrow{n}}(w')$. Following Theorem~\ref{T:RewritingPropertiesYoung}, we have~$\nf([w]_s,\Fr\Sr_n) = \nf([w']_s,\Fr\Sr_n)$, and then~$w\approx_{\Pl_n} w'$ by Theorem~\ref{T:MorphismRect}.

Prove Condition~{\bf ii)}.
Following Theorem~\ref{T:RewritingPropertiesYoung}, for any~$d$ in~$\dSkew{n}$, we have~$\rect(d) = C_{\Yrow{n}}(R_{SW}(d))$, showing the claim.
\end{proof}

\begin{corollary}[commutation of insertion algorithms]
\label{C:Consequense2}
\begin{enumerate}[{\bf i)}] 
\item For all $d$ in $\dSkew{n}$ and $x$ in $[n]$, we have~$\rect(d\insr{I_r^a} x)= \rect(d)\insr{S_r}x$ and~$\rect(x \insl{I_l^a}d)= x\insl{S_l}\rect(d)$.
\item The insertion algorithms~$S_r$ and $S_l$ commute, that is the following equality
\[
y \insl{S_l} (t\insr{S_r} x)
\: = \:
(y \insl{S_l} t) \insr{S_r} 
\]
holds in~$\Young{n}$,  for all $t$ in $\Young{n}$ and $x$ in $[n]$. 
\end{enumerate}
\end{corollary}

\begin{proof}
Prove Condition~{\bf i)}.
By Theorem~\ref{T:MorphismRect}, we have~$R_{SW}(\rect(d\insr{I_r^a} x))\approx_{\Pl_n}  R_{SW}(d\insr{I_r^a}x)$. Moreover, we have~$R_{SW}(d\insr{I_r^a}x)= R_{SW}(d)x$, hence~$R_{SW}(\rect(d\insr{I_r^a} x))\approx_{\Pl_n}   R_{SW}(d)x$ . On the other hand, the following equalities holds in $\Young{n}$:
\[
\rect(d)\insr{S_r}x = C_{\Yrow{n}}(R_{SW}(\rect(d))x) = \nf([R_{SW}(\rect(d))x]_s,\Fr\Sr_n).
\]
Then~$\rect(d)\insr{S_r} x \approx_{\Fr\Sr_n} [R_{SW}(\rect(d))x]_s$, and thus by Theorem~\ref{T:MorphismRect}, we deduce 
\[
R_{SW}(\rect(d)\insr{S_r}x)\approx_{\Pl_n}  R_{SW}(\rect(d))x \approx_{\Pl_n}  R_{SW}(d)x,
\]
showing that~$R_{SW}(\rect(d\insr{I_r^a} x))\approx_{\Pl_n}  R_{SW}(\rect(d)\insr{S_r}x)$. Finally, following the cross-section property, we obtain~$\rect(d\insr{I_r^a} x)= \rect(d)\insr{I_r^a}x$. Similarly, we show that~$\rect(x \insl{I_l^a}d)= x\insl{S_l}\rect(d)$.

Prove Condition~{\bf ii)}.
Following Condition~{\bf i)}, we have
\[
\rect(y \insl{I_l^a} (d \insr{I_r^a} x )) = y \insl{S_l} (\rect(d \insr{I_r^a} x)) = y \insl{S_l} (\rect(d)\insr{S_r} x),
\]
and
\[
\rect(y \insl{I_l^a} (d \insr{I_r^a} x )) = \rect((y \insl{I_l^a} d)) \insr{S_r} x =(y \insl{S_l} \rect(d)) \insr{S_r} x,
\]
for all $d$ in $D$ and $x,y$ in $[n]$.
By commutation of $I_r^a$ and $I_r^a$, we deduce the following equality
\[
y \insl{S_l} (\rect(d) \insr{S_r} x )  
\:=\:
(y \insl{S_l} \rect(d)) \insr{S_r} x.
\]
The map $\rect$ being surjective, we deduce that $S_r$ and $S_l$ commute.
\end{proof}

\begin{corollary}[morphism of monoids]
\label{C:Consequense3}
 The map~$\rect$ induces a morphism of monoids between~$(\dSKrow{n}, \star_{I_r^a})$ and~$(\Yrow{n},\star_{S_r})$.
\end{corollary}

\begin{proof}
Following Condition~{\bf i)} of Corollary~\ref{C:Consequense2}, we first prove that $\rect(d\insr{I_r^a}u) =\rect(d)\insr{S_r}u$, for all $d$ in~$\dSkew{n}$ and $u$ in~$[n]^\ast$ by induction on $|u|$. Suppose the equality holds when $|u|=k-1$, then for $y$ in $[n]$ we have
\[
\begin{array}{rl}
\rect(d\insr{I_r^a}uy) 
= \rect((d\insr{I_r^a}u)\insr{I_r^a}y) 
&= \rect((d\insr{I_r^a}u))\insr{S_r}y\\
&=(\rect(d)\insr{S_r}u)\insr{S_r}y) 
= \rect(d)\insr{S_r}uy.
\end{array}
\]
In an other hand, following Condition~{\bf ii)} of Corollary~\ref{C:Consequense1}, we have $C_{\Yrow{n}}(R_{SW}(d')) =C_{\Yrow{n}}( R_{SW}(\rect(d')))$, for all $d'$ in~$\dSkew{n}$, hence $C_{\Yrow{n}}(R_{SW}(\rect(d))R_{SW}(d'))=C_{\Yrow{n}}(R_{SW}(\rect(d))R_{SW}(\rect(d')))$.
 As a consequence, we have
\[
 \begin{array}{rl}
\rect(d\star_{I_r^a} d') 
= \rect(d\insr{I_r^a}R_{SW}(d') 
&= \rect(d)\insr{S_r}R_{SW}(d')\\
&=\rect(d)\insr{S_r}R_{SW}(\rect(d') 
= \rect(d)\star_{S_r} \rect(d'),
\end{array}
\]
for all~$d,d'$ in $\dSkew{n}$, showing the claim.

\end{proof}

\subsubsection{Remark} Note that Schensted's insertion algorithms are related to the jeu de taquin by the following formulas
\[
t \insr{S_r} x \;=\; \rect\big([R_{SW}(t)]_s\insr{I_{r}^{a}} x\big),
\qquad\quad
x \insl{S_l} t \;=\; \rect\big(x \insl{I_l^a}[R_{SW}(t)]_s \big),
\]
for all $t$ in~$\Young{n}$ and $x$ in $[n]$.
Note also that the associativity of~$\star_{S_r}$ is also deduced from the morphism~$\pi_{tq}$. Indeed, for all $t$ in~$\Young{n}$, and $x$ in $[n]$, we have
$
t \insr{S_r} x = 
\rect(t|_1 \ytableausetup{smalltableaux}
\begin{ytableau}
x
\end{ytableau})
$, and thus~$
t\star_{S_r} t'
\:=\:\rect(t |_s t')
$, for all~$t,t'\in \Young{n}$.
By Theorem~\ref{T:RewritingPropertiesYoung},  we obtain
\[
(t\star_{S_r} t') \star_{S_r} t'' 
\:=\: 
\rect(t |_s t' |_s t'') 
\:=\: 
t \star_{S_r} (t' \star_{S_r} t''),
\]
for any~$t,t',t''\in \Young{n}$.

\section{Conclusion and perspectives}
In this article, we have introduced the notion of string of columns rewriting system as rewriting systems over glued sequences of columns.
This gives a rewriting framework  to prove the confluence of   Schützenberger's jeu de taquin algorithm defined on the structure of tableaux.  Our construction leads us to formulate several perspectives:
\begin{itemize}
\item [$\bullet$]
In~\cite{HageMalbos17}, we make explicit the relations among the relations of the Knuth relations for the plactic monoid of type A.
We expect that the rewriting presentation of the jeu de taquin introduced in this article could make explicit the relations among the relations for the cross-section property on sequence of columns  for the plactic monoid of type~A. 
Such a study of the relations among the relations for the presentations of   a monoid constitutes the first step in an explicit construction of a cofibrant approximation of the monoid in the category of~$(\omega,1)$-categories and of actions of the monoid on categories, see~\cite{GaussentGuiraudMalbos15, GuiraudMalbos12advances,  HageMalbos17, HageMalbos22}.

\item [$\bullet$]  Schützenberger's jeu de taquin gives a proof of the Littlewood--Richardson rule which is a combinatorial description of the coefficients that arise when decomposing a product of two Schur polynomials as a linear combination of other Schur polynomials. In particular, these coefficients count certain types of skew tableaux that are rectified by the jeu de taquin to Young tableaux. We except that these coefficients could be also described by a rewriting approach using a zigzag sequences of reductions from certain types of skew tableaux to their normal forms with respect to the rewriting  presentation of the jeu de taquin introduced in this article.

\item [$\bullet$] The construction applied in this article on the jeu de taquin could be also applied on similar algorithms defined on other structures of tableaux,~\cite{Lecouvey02, Lecouvey03, ThomasYoung09, Tewari2015, RomikSniady15, Hage2021Super}. 
In particular, we expect constructive proofs by rewriting of the properties relating these algorithms to the plactic monoids of classical types and the super plactic monoid which are respectively related to the representations of the finite dimensional semisimple Lie algebras of classical types and the general Lie superalgebra.

\end{itemize}

\begin{small}
\renewcommand{\refname}{\Large\textsc{References}}
\bibliographystyle{plain}
\bibliography{biblioCURRENT}
\end{small}

\quad

\vfill

\begin{flushright}
\begin{small}
\noindent \textsc{Nohra Hage} \\
\url{nohra.hage@univ-catholille.fr} \\
Faculté de Gestion, Economie \& Sciences (FGES),\\
Université Catholique de Lille,\\
60 bd Vauban, \\
CS 40109, 59016 Lille Cedex, France\\

\bigskip

\noindent \textsc{Philippe Malbos} \\
\url{malbos@math.univ-lyon1.fr} \\
Univ Lyon, Universit\'e Claude Bernard Lyon 1\\
CNRS UMR 5208, Institut Camille Jordan\\
43 blvd. du 11 novembre 1918\\
F-69622 Villeurbanne cedex, France
\end{small}
\end{flushright}

\vspace{0.25cm}

\begin{small}---\;\;\today\;\;-\;\;\hhmm\;\;---\end{small} \hfill
\end{document}